\documentclass[12pt]{article}
\usepackage[english]{babel}
\usepackage{amsmath,amsthm,amsfonts,amssymb,epsfig,verbatim, graphicx}
\usepackage[left=1in,top=1in,right=1in]{geometry}
\newtheorem{thm}{Theorem}[section]
\newtheorem{lem}[thm]{Lemma}
\newtheorem{prop}[thm]{Proposition}
\newtheorem{cor}[thm]{Corollary}
\newtheorem{df}[thm]{Definition}

\newcommand{\E}{\mathbb{E}}

\newcommand{\R}{\mathbb{R}}

\newcommand{\N}{\mathbb{N}}

\def\<{\langle}
\def\>{\rangle}

\newcommand{\Pro}{\ensuremath{\mathbb{P}}}

\newcommand{\Var}{\mathop{\mathrm{Var}}\nolimits}
\newcommand{\Cov}{\mathop{\mathrm{Cov}}\nolimits}
\newcommand{\CC}{C}
\theoremstyle{usual}
\newtheorem{theorem}{Theorem}[section]

\newtheorem{proposition}{Proposition}%[section]
\newtheorem{clam}{Claim}

\newtheoremstyle{likedef}
  {}%
  {}%
  {}%
  {\parindent}%
  {\bfseries}%
  {.}%
  {.5em}%
  {}%

\theoremstyle{likedef}

\newtheorem{remark}{Remark}%[section]
\newtheorem{example}{Example}%[section]

\numberwithin{equation}{section}

\title{Rate of convergence of the mean for sub-additive ergodic sequences}
\author{Antonio Auffinger \thanks{The research of A. A. is supported by NSF grant DMS-1407554.}  \\ \small{University of Chicago} \and Michael Damron \thanks{The research of M. D. is supported by NSF grant DMS-1419230.} \\ \small{Indiana University} \and Jack Hanson\\ \small{Indiana University}}
\date{\today}

\begin{document}
\maketitle
\begin{abstract}
For sub-additive ergodic processes $\{X_{m,n}\}$ with weak dependence, we analyze the rate of convergence of $\mathbb{E}X_{0,n}/n$ to its limit $g$. We define an exponent $\gamma$ given roughly by $\mathbb{E}X_{0,n} \sim ng + n^\gamma$, and, assuming existence of a fluctuation exponent $\chi$ that gives $\Var X_{0,n} \sim n^{2\chi}$, we provide a lower bound for $\gamma$ of the form $\gamma \geq \chi$. The main requirement is that $\chi \neq 1/2$. In the case $\chi=1/2$ and under the assumption $\Var X_{0,n} = O(n/(\log n)^\beta)$ for some $\beta>0$, we prove $\gamma \geq \chi - c(\beta)$ for a $\beta$-dependent constant $c(\beta)$. These results show in particular that non-diffusive fluctuations are associated to non-trivial $\gamma$. Various models, including first-passage percolation, directed polymers, the minimum of a branching random walk and bin packing, fall into our general framework, and the results apply assuming $\chi$ exists. In the case of first-passage percolation in $\mathbb Z^d$, we provide a version of $\gamma \geq -1/2$ without assuming existence of $\chi$.
\end{abstract}

\section{Introduction}

\subsection{Subadditive ergodic theorem}
Sub-additive ergodic theory plays a major role in modern mathematics. Its development began in probability with the work of Hammersley-Welsh \cite{HW} and with the seminal paper of Kingman \cite{Kingman}. Substantial discussion with many examples and applications in diverse areas can be found in \cite{Cafarelli, Hamm, KingmanNotes, Liggett, Logan, Smy, Walters} and the references therein. 

At the heart of the theory sits Kingman's sub-additive ergodic theorem \cite[Theorem~5]{Kingman}. It establishes the existence of an almost sure and $L_1$ limit of $X_{0,n}/n$  for a sub-additive ergodic process $\{X_{m,n}\}$. In probabilistic language, this theorem is a type of law of large numbers. However, it differs significantly from the independent, identically distributed case. In particular, even in the presence of weak dependence, the rate of convergence can be difficult to control. 

This rate is the main subject of this paper. We will see that the corresponding error term can be broken into two pieces, associated respectively to random and deterministic fluctuations of the sequence. Our goal will be to relate these two errors, and the main theorem that we present will establish such a relation valid in a wide range of applications.

\subsection{The rate of convergence}\label{sec: setting}

We now discuss the background and history of the rate of convergence, beginning with the formal setup. Suppose that $\{X_{m,n}\}$ is a sequence of random variables indexed by nonnegative integers $m < n$. Assume the following conditions:
\begin{enumerate}
\item $X_{0,n} \leq X_{0,m} + X_{m,n}$ for all $0 \leq m < n$;
\item $\{X_{m+1,n+1}:\, 0 \leq m < n\} \stackrel{d}{=} \{X_{m,n}:\, 0 \leq m < n\} $ for each $n$, and this shift operation is ergodic;
%\item There is some $\delta > 0$ such that $\E \left| X_{0,n}\right| ^{2+\delta}$ for each $n$;
\item $\mathbb{E} X_{0,n} > -c n$ for some $c > 0$ and all $n$.
\end{enumerate}
These assumptions are adequate for Kingman's sub-additive ergodic theorem \cite{Liggett}, which gives
\[\lim_{n \rightarrow \infty} \frac{X_{0,n}}{n} = \lim_{n \rightarrow \infty} \frac{\mathbb{E} X_{0,n}}{n} = \inf_{n \geq 1} \frac{\mathbb{E}X_{0,n}}{n} =: g,\]
where the convergence on the left-hand side is a.s. and in $L^1$. Although Kingman's theorem gives its existence, the value of the limit is process-dependent, and precise characterization is often a challenging question. 

The rate of convergence of $X_{0,n}/n$ to $g$ has two distinct sources: the deviation of $X_{0,n}$ from its mean $\E X_{0,n}$ and the deviation of $\E X_{0,n}$ from $ng$. Precisely, if we write $X_{0,n} = ng + o(n)$ then 
\begin{equation}\label{eq:brackets}
o(n) = \underbrace{X_{0,n} - \mathbb{E}X_{0,n}}_{random~fluctuations} + \underbrace{\mathbb{E}X_{0,n} - ng}_{non-random~fluctuations} \, .
\end{equation}
One can typically control the random fluctuation term by using techniques from concentration of measure \cite{BLM, Ledoux}. For example, in many probabilistic models that we will discuss, exponential or Gaussian concentration bounds are available. This is particularly true when there is an underlying product space and the $X_{m,n}$'s  do not depend too heavily on any of the coordinates \cite{Talagrand}.

\subsection{Main goals}

A priori, there is no clear relation between two parts of \eqref{eq:brackets}. Indeed, in many examples we can modify the long term behavior of the non-random fluctuation term by adding a deterministic sub-additive process and therefore not changing the random part. However, the only current methods to bound the non-random fluctuation term use estimates for the random fluctuation part and  rely on intrinsic properties of the process.

To put our theorems in context, we briefly review past results. Upper bounds for the non-random fluctuations have been established in many examples including first-passage percolation \cite{Alexander, Alexander2, AlexanderDirected, Chatterjee2013, Kesten}, last-passage percolation \cite{Jinho, Joha} and longest common subsequence \cite{Alexander}. In 1997, a general theory was given by  K. Alexander to bound the non-random fluctuation term by the random fluctuation term \cite{Alexander}. His ingenious convex hull approximation property applies to several processes that are naturally indexed by the $d$-dimensional integer lattice.

Lower bounds for non-random fluctuations are significantly  less developed. To our knowledge, the only results in this direction are given by Kesten \cite{Kesten} in the example of first-passage percolation and \cite{Jaillard, KL} and \cite[Section~8.2]{Yukich} for classes of ``Euclidean functionals'' of point processes. In particular, there is no general framework that deals with lower bounds. Our objective is to provide such a framework. 

 Below we list the main points of the article. Our overall assumption is that the process $\{X_{m,n}\}$ has a  weak-dependence structure. In points 1 and 2 we assume the existence of a random fluctuation exponent. 
\begin{enumerate}
\item Non-diffusive fluctuations give rise to non trivial non-random fluctuations. In our main theorem, Theorem \ref{thm:main1}, we show that if $\Var X_{0,n} = O(n/ (\log n)^{\beta} )$ then $\E X_{0,n} -ng$ is at least of order of $n^c$ for some constant $c$. A similar result holds if the variance grows as a super-linear power of $n$.
\item In ideal non-diffusive conditions both fluctuation terms have the same order. In the same theorem, Theorem \ref{thm:main1}, we show that if $\Var X_{0,n}$ behaves as $n^{2 \chi}$ for $\chi \neq 1/2$ then  $\E X_{0,n} -ng$ grows at least as $n^{\chi}$. Comparing to the work of Alexander mentioned above, one therefore obtains that in many applications of interest both terms are comparable. This result is particularly suited to models of KPZ type \cite{KPZ}, where $\chi$ is expected to be $1/3$. 
 
\item Application to $d$-dimensional first-passage percolation. We establish a lower bound of the form $\E X_{0,n} -ng \geq n^{-1/2-\epsilon}$ for infinitely many $n$ and any $\epsilon>0$ without assuming the existence of any fluctuation exponent. The best bound before this work is due to Kesten \cite{Kesten} who showed a lower bound of order $n^{-1}$ for all $n$. This bound becomes more relevant in high dimensions, where many believe the random fluctuation exponent approaches zero. %Our bound is not far off from this prediction. 

\item The assumption of weak-dependence is easily checked in many models of interest. We provide  examples in Section \ref{sec:examples}. 

\end{enumerate}

The proofs of these results come from a new iterative scheme that we develop in Section~\ref{sec:proof} to bound the ratios $\Var X_{0,n}/ \Var X_{0,m}$ for $n$ and $m$ large. We will sketch this procedure in Section \ref{sec:sketch}.

\section{Main results}

\subsection{Weak-dependence and exponents}

The main results on convergence that follow require some asymptotic independence for the collection $\{X_{m,n}\}$. One possibility is to assume a strong mixing rate for this collection. In practice, though, these random variables often obey a positive association which allows us to assume merely suitable decay of correlations. So we add the following to our list of assumptions.
\begin{itemize}
\item[4.] The collection $\{X_{m,n}\}$ obeys either {\bf Condition I} (mixing) or {\bf Condition II} (association with small covariances) detailed below.
\end{itemize}
This assumption is valid, for instance, if for each $n \geq 1$, the variables $X_{0,n}, X_{n,2n}, \ldots$ are independent with a non-degenerate distribution and two moments. (Association condition II holds in this case.) Let us now detail these conditions.

%We can now define our two asymptotic independence assumptions.

For each $n \geq 1,$ define the function $\alpha_n: \mathbb{Z}_+ \rightarrow \mathbb{R}$ by
\begin{equation}
\label{eq:alphadef}
\alpha_n(m) = \sup_{k \geq 0} \sup_{A,B} \left| \mathbb{P}(A \cap B) - \mathbb{P}(A) \mathbb{P}(B) \right|,
\end{equation}
where the second supremum is over all events $A \in \sigma(X_{in, (i+1)n}:\, 0 \leq i < k)$ and $B \in \sigma(X_{in,(i+1)n}:\, k+m \leq i )$.
This is a version of a strong mixing coefficient.

\medskip
\noindent
{\bf Condition I}: (Mixing condition.) There exist positive constants $\CC_4,\CC_5$ such that, for each $n$,
\begin{enumerate}
\item $\Var \sum_{i \in \mathcal{B}} X_{(i-1)n,in} \leq \CC_4 b ~\Var X_{0,n}$ for all $b$ and sets $\mathcal{B}$ of indices such that $\#\mathcal{B}=b$;

\item $\CC_5 b ~\Var X_{0,n}\leq \Var \sum_{i=1}^b X_{(i-1)n,in}$ for all $b$;

\item for each $\kappa>0$, there exists $\CC_6$ such that, for all $n$ and $x$, $\alpha_n(x) \leq \CC_6 x^{-\kappa}$.
\end{enumerate}

Item 3 above is a strong example of a sufficient mixing condition and can be considerably relaxed. Item 1 is a covariance assertion, and can be proved from 3 in the presence of finite $2+\delta$ moments if the mixing rate $\alpha_n$ is assumed to increase suitably with $n$. 

\medskip
\noindent
{\bf Condition II}: (Association condition.) The variables $\{X_{m,n}\}$ are positively associated: for any real-valued coordinatewise nondecreasing functions $\Phi,\,\Psi$ on realizations of $\{X_{m,n}\}$ \[\mathrm{Cov}(\Phi ,  \Psi) \geq 0 \quad \text{ if it exists}. \]
In addition, we need to assume an asymptotic negligibility of the covariances. The condition comes from \cite{CG}: there is a function $u : \{0,1,2, \ldots\} \to \mathbb{R}$ such that $u(r) \to 0$ as $r \to \infty$ and
\begin{equation}
\label{eq:covcond}
\frac{\max_{i \geq 0} \sum_{j : |i-j| \geq r} \mathrm{Cov} (X_{in,(i+1)n}, X_{jn,(j+1)n})}{\Var X_{0,n}} \leq u(r) \text{ for all } n\ .
\end{equation}
The covariance assumption \eqref{eq:covcond} can often be verified if one has bounds on the mixing rate $\alpha_n$. In this case one can use Lemma~\ref{lem:orthogcontrol} to decouple $X_{in,(i+1)n}$ from $X_{jn,(j+1)n}$. 

%We now describe in detail condition 4, which gives some asymptotic independence of the array $\{X_{m,n}\}$.

%The main purpose of condition 4  is to give us a central limit theorem for the sum of the first $k_n$ variables $X_{0,n} + \cdots + X_{(k_n-1)n, k_n n}$, where $k_n$ is assumed to diverge with $n$. %The two conditions listed below can be replaced by many others. They are examples that allow us to prove the main estimate of Proposition~\ref{prop: main_estimate}.

\bigskip \noindent {\bf Definition of exponents} 

We will define two exponents: one measures the random fluctuations of $\{X_{m,n}\}$ and the other measures the convergence rate of $\mathbb{E}X_{0,n}/n$ to $g$. The main results give a relation between the two.

%\begin{rem}
%The above definitions can be rephrased as
%\[
%\overline{\chi}_p = \limsup_{n\to \infty} \frac{\log \|X_{0,n} - \mathbb{E}X_{0,n}\|_p}{\log n}
%\]
%and
%\[
%\underline{\chi}_p = \liminf_{n\to \infty} \frac{\log \|X_{0,n} - \mathbb{E}X_{0,n}\|_p}{\log n}\ .
%\]
%\end{rem}
\begin{df}\label{def:exponentgamma}
The exponents $\underline{\gamma}$ and $\overline{\gamma}$ are defined as
\[
\underline{\gamma} = \liminf_n \frac{\log (\mathbb{E}X_{0,n}-n g)}{\log n} \text{ and } \overline{\gamma} = \limsup_n \frac{\log (\mathbb{E}X_{0,n}-n g)}{\log n}\ .
\]
Here, $\log 0$ is defined as $-\infty$.
\end{df}

Similarly we define fluctuation exponents. For a random variable $Z$ and $p>0$, write $\|Z\|_p = \left( \mathbb{E}|Z|^p \right)^{1/p}$.

\begin{df}\label{def:exponentxi}
For $p>0$, the fluctuation exponents $\underline{\chi}_p$ and $\overline{\chi}_p$ are defined as
\[
\underline{\chi}_p = \liminf_{n\to \infty} \frac{\log \|X_{0,n} - \mathbb{E}X_{0,n} \|_p}{\log n} \text{ and } \overline{\chi}_p = \limsup_{n\to \infty} \frac{\log \|X_{0,n} - \mathbb{E}X_{0,n}\|_p}{\log n}\ .
\]
\end{df}

\noindent
Note that $\underline{\chi}_p \leq \overline{\chi}_p$ and by Jensen,
\[
\underline{\chi}_p \leq \underline{\chi}_q \text{ and } \overline{\chi}_p \leq \overline{\chi}_q \text{ if } p \leq q\ .
\]

\begin{remark} The definitions above only specify the first order growth of the quantities $ \E X_{0,n}-ng$ and $\| X_{0,n} - \E X_{0,n} \|_p$. For instance, if there exist positive constants $c, C$ such that 
\[
c n^{2\chi} \leq \Var X_{0,n} \leq C n^{2\chi} \text{ for all large }n\ ,
\]
then $\underline{\chi}_2 = \overline{\chi}_2 = \chi$, and similarly for $p \neq 2$ and $\underline{\gamma}, \overline{\gamma}$.
\end{remark}

\subsection{Main theorem}

We are now ready to state the main result of this paper. We give the following relation between the exponents defined above. Define $\Lambda = 2/\CC_5$ under mixing condition I (with $\CC_5$ from item 2 of that condition) and $\Lambda = 2$ under association condition II.

\begin{thm}\label{thm:main1}
Assume 1-4 and $\mathbb{E}|X_{0,n}|^{2+\delta}<\infty$ for some $\delta>0$ and all $n$.
\begin{enumerate}
\item If $\chi := \underline{\chi}_2 = \overline{\chi}_{2+\delta}<\infty$, then 
\[
\begin{cases}
\underline{\gamma} \geq \chi & \text{ if } \chi \neq 1/2 \\
\overline{\gamma} \geq \chi - \frac{1}{2}(\Lambda^{\beta^{-1}}-1) & \text{ if } \Var X_{0,n} = O\left( \frac{n}{(\log n)^\beta} \right) \text{ for some }\beta>0 \\
\end{cases}\ .
\]
\item If $\chi := \overline{\chi}_2=\overline{\chi}_{2+\delta}<\infty$, then if $\chi < 1/2$, $\overline{\gamma} \geq \chi$.
\end{enumerate}
\end{thm}

%{\tt Discuss assumption $\underline{\chi}_2 = \overline{\chi}_{2+\delta}$ here.}

\begin{remark}

If $X_{m,n} = \sum_{k=m+1}^n X_k$, where $(X_k)$ is a given i.i.d. sequence with $2+\delta$ moments, then $\chi=1/2$ but $\underline{\gamma} = \overline{\gamma} = -\infty$. The assumption $\Var X_{0,n} = O\left(n/(\log n)^\beta\right)$ for some $\beta>0$ ensures sub-linear variance and rules out such i.i.d. sums. Note that if this condition holds for all $\beta>0$ then we obtain the result $\overline{\gamma} \geq \chi$.
\end{remark}

\begin{remark}

The assumption $\underline{\chi}_2 = \overline{\chi}_{2+\delta}$ is satisfied if there exists some sequence of constants $b_n$ with $(\log b_n)/\log n \to \chi$ such that for some non-degenerate $Z$,
\[
b_n^{-1}(X_{0,n} -  \E X_{0,n}) \to Z \text{ in } L^{2+\delta} \text{ as } n \to \infty\ .
\] 
This is expected in all examples that we discuss in Section \ref{sec:examples}.
\end{remark}

The case $\chi\neq 1/2$ of Theorem~\ref{thm:main1} can also be handled with moment assumptions instead of assumption~4. We state this result as a separate theorem. Setting $X_{0,n}' = X_{0,n}-\mathbb{E}X_{0,n}$, the relevant moment conditions are:
\begin{enumerate}
\item[(M1). ] For some $\CC_1>0$, $\Var \sum_{k=1}^b X_{(k-1)n,kn} \leq \CC_1 b \Var X_{0,n}$ for all $n$ and $b$.
\item[(M2). ] For some $\CC_2,\CC_3$ and $\delta>0$, both of the following hold.
\begin{enumerate}
\item $\Var \sum_{k=1}^b X_{(k-1)n,kn} \geq \CC_2 b \Var X_{0,n}$ for all $n$ and $b$ and
\item $\mathbb{E}\left| \sum_{k=1}^b X_{(k-1)n,kn}'\right|^{2+\delta} \leq \CC_3 b^{1+\frac{\delta}{2}} ~\mathbb{E}|X_{0,n}'|^{2+\delta}$ for all $n$ and $b$.
\end{enumerate}
\end{enumerate}
(M1) and (M2a) are implied by either mixing condition I or association condition II. Furthermore, all of these conditions hold if for each $n$, $X_{0,n}, X_{n,2n}, \ldots$ are independent with $2+\delta$ moments.

\begin{thm}\label{thm: shorter}
Assume 1-3 and $\mathbb{E}|X_{0,n}|^{2+\delta}<\infty$ for some $\delta>0$ and all $n$. Assume that $\chi := \underline{\chi}_2 = \overline{\chi}_{2+\delta}<\infty$. 
\begin{enumerate}
\item If (M1) holds and $\chi > 1/2$ then $\underline{\gamma} \geq \chi$.
\item If (M2) holds and $\chi < 1/2$, then $\underline{\gamma} \geq \chi$.
\end{enumerate}
\end{thm}

\begin{remark}\label{rem: fancy}
If instead $\chi:=\overline{\chi}_2 = \overline{\chi}_{2+\delta} < 1/2$ for some $\delta>0$, then the above result holds with $\underline{\gamma}$ replaced by $\overline{\gamma}$.
\end{remark}

\begin{remark}
The proof of item 2 can be adapted to the case that (M2) holds with $1+\delta/2$ replaced by $\sigma$ for any $\sigma < 1+\frac{\delta}{2} + \frac{\delta}{2}(1/2-\chi)$.
\end{remark}

\section{Examples and a  counter-example}\label{sec:examples}
\subsection{A counter-example}
We start this section by showing that without any form of asymptotic independence (either assumption 4 or appropriate moment conditions), one should not expect the result of Theorem \ref{thm:main1} to hold. This can be seen by the following example. 
\begin{example}[Arbitrary $\chi$ and $\gamma$] Fix constants $H \in (0,1)$ and $h \in (-\infty, 1)$. We construct a sequence $\{X_{m,n}\}$ that satisfies assumptions 1, 2 and 3 and has exponents
$$ \chi = H, \quad  \gamma:= \overline \gamma = \underline \gamma = h.$$

Let $Y_t, t \in \R,$ be a fractional Brownian motion with Hurst exponent $H$. $Y_t$ is a continuous time Gaussian process with stationary increments that satisfies
$$ \E Y_t = 0  \quad \text{ and } \quad  \E Y_t Y_s = \frac{1}{2}\big(|t|^{2H} +|s|^{2H} - |t-s|^{2H}\big).$$
For $0\leq m<n$, we set $Y_{m,n} := Y_{n}-Y_{m}$. The array $\{ Y_{m,n}\}$ is additive, stationary under shifts, and has mean zero, so it satisfies assumptions 1-3. Furthermore, $\Var Y_{0,n} = \E Y_{0,n}^2= n^{2H}.$

Now, let $(x_n)$ be a sub-additive sequence of non-negative real numbers such that $x_n/n \rightarrow 0.$ We can choose the convergence of $x_n/n$ as slow as we want. Indeed, let $x_n \geq 0$ with $x_n/n$ decreasing to $0$.
We see that 

$$ x_{n+m} = m\frac{x_{n+m}}{n+m}+n\frac{x_{n+m}}{n+m} \leq m\frac{x_m}{m} + n\frac{x_n}{n}= x_m+x_n.$$

For $0\leq m<n$, define $$X_{m,n} = Y_{m,n} + x_{n-m}.$$
The array $\{ X_{m,n}\}$ satisfies 1, 2 and 3 and $\E X_{0,n}/n \to 0$. However we have  $\chi = H \in (0,1)$ and $\E X_{0,n}= x_n$ so $\gamma$ can be taken as any number in $(-\infty,1)$. 
\end{example} 

A word of comment is needed here. If one tries to build an example as above requiring the process $Y_t$ to have independent stationary increments, then one does not violate Theorem~\ref{thm:main1}. Indeed, a computation shows that the covariance of  the process $Y_t$ must satisfy  $\E Y_t Y_s~=~\sigma \min\{s,t\}$ for some $\sigma>0$ and thus $\chi =1/2$. This example stresses once more that the assumption of anomalous diffusion is necessary in Theorem~\ref{thm:main1}. 

%{\tt stress $\chi=1/2$ case as having mixing but $\gamma$ can be anything -- this shows we cannot take $\Var X_{0,n} \sim n$ in theorem}

%Our next theorem covers the case where $\chi = 1/2$. In this case, we should not expect a version of $\gamma$ to be large when $\chi = 1/2$ (without extra assumptions), since it is not true for a sum of i.i.d. variables.
%\begin{thm}\label{thm1.2}
%Assume $\chi:=\underline{\chi}_2 = \overline{\chi}_{2q} = 1/2$ for some $q  \in (1, (4\alpha - 1)/4]$ where $\alpha$ satisfies \eqref{eq:Assumptionmoments} and that
%\[
%\Var T(0,n\mathbf{e}_1) = o\left( \frac{n}{\log n} \right)\ .
%\]
%Then $\overline{\gamma} \geq \chi$. 
%\end{thm}

We now turn our attention to examples that satisfy our main hypothesis. 
\subsection{First-Passage Percolation}

Let $(t_e)_{e \in \mathcal{E}^d}$ be a collection of nonnegative, i.i.d. random variables assigned to the nearest-neighbor edges $\mathcal{E}^d$ of the integer lattice $\mathbb{Z}^d$. In first-passage percolation, we consider the pseudo-metric induced by these weights. Namely, the \emph{passage time} between vertices $x,y \in \mathbb{Z}^d$ is defined
\[
T(x,y) = \inf_{\pi : x \to y} T(\pi)\ ,
\]
where $\pi$ is a lattice path from $x$ to $y$ (a sequence $x=x_0, \ldots, x_n=y$ of vertices such that $\|x_k-x_{k+1}\|_1=1$ for $k=0, \ldots, n-1$) and $T(\pi) = \sum_{k=0}^{n-1} t_{\{x_k, x_{k+1}\}}$ is the sum of weights along $\pi$. As usual, we will assume
\begin{equation}\label{eq: percolation_assumption}
\mathbb{P}(t_e=0)<p_c(d)\ ,
\end{equation} 
where $p_c(d)$ is the critical probability for $d$-dimensional bond percolation.

Given a vertex $x \in \mathbb{Z}^d$, the sequence
\begin{equation}\label{eq: passage_def}
\{X_{m,n}\} = \{T(mx,nx) : m,n \in \mathbb{N},~ 0\leq m < n\}
\end{equation}
is a sub-additive process and given that $\E X_{0,1} < \infty$ it satisfies assumptions 1-3 from Section~\ref{sec: setting}.  The \emph{time constant} $g(x)$ is defined as $g(x) = \lim_n \frac{\mathbb{E}T(0,nx)}{n}$.  We will show that this model satisfies condition 4 in section \ref{sec:fpp}.

Our next theorem gives the bound $\overline{\gamma} \geq -1/2$ under minimal assumptions. This result should be compared to \cite[Theorem~1]{Kesten}, where it is shown that $\underline{\gamma} \geq -1$.
\begin{thm}\label{thm: second_fpp_thm}
Assume \eqref{eq: percolation_assumption}, that the distribution of $t_e$ is not concentrated at a point, and that 
\begin{equation}\label{eq: exp_assumption}
\mathbb{E}e^{\alpha t_e} < \infty \text{ for some } \alpha>0\ .
\end{equation}
One has the bound $\overline{\gamma} \geq -1/2$: for any nonzero $x \in \mathbb{Z}^d$ and $\epsilon>0$,
\[
\mathbb{E}T(0,nx) - ng(x) \geq n^{-\frac{1}{2}-\epsilon} \text{ for infinitely many } n\ .
\]
\end{thm}

Theorem~\ref{thm: second_fpp_thm} is proved in Section~\ref{sec: second_fpp}.

\begin{remark}
Alexander \cite{AlexanderDirected} has remarked (see also a proof in Chatterjee \cite{Chatterjee2013}) that if $\hat \chi$ is any number such that for some $a>0$,
\[
\mathbb{P}\left( |T(0,y) - \mathbb{E}T(0,y)| \geq \lambda \|y\|_1^{\hat \chi} \right) \leq e^{-a \lambda} \text{ for all } \lambda\geq 0,~ y \in \mathbb{Z}^d\ ,
\]
then $\overline{\gamma} \leq \hat \chi$. Note that if this exponential inequality holds for some $\hat \chi$, then $\overline{\chi}_p \leq \hat \chi$ for all $p>0$. Combining these observations with Theorem~\ref{thm: main_fpp_thm}, if $\hat \chi$ can be taken to be $\chi: = \hat \chi = \underline{\chi}_2$, then $\gamma:= \underline{\gamma} = \overline{\gamma} = \chi$ when $\chi<1/2$ and $\overline{\gamma} = \chi$ under the assumption $\Var T(0,nx) = O(n/(\log n)^\beta)$ for every $\beta>0$.
\end{remark}

\begin{remark}
Under the assumption $\mathbb{E}t_e^2 (\log t_e)_+<\infty$, it was shown by Damron-Hanson-Sosoe \cite[Theorem 1.1]{DHS} that $\Var T(0,nx) = O(n/\log n)$. Combining this with Theorem~\ref{thm: main_fpp_thm}, if $\underline{\chi}_2 = \overline{\chi}_{2+\delta}$ for some $\delta>0$ and \eqref{eq: percolation_assumption} holds, then $\overline{\gamma} \geq 0$.
\end{remark}

\subsection{Directed first-passage percolation}

The only difference between FPP and directed FPP is the constraint that all paths under consideration are directed; that is, for $x,y \in \mathbb{Z}^d$, write $y \leq x$ if this inequality holds coordinate-wise, and a path $y=x_0, x_1, \ldots, x_n =x$ is directed if $x_i \geq x_{i-1}$ for $i=1, \ldots, n$. We again place i.i.d. nonnegative passage times $(t_e)$ on $\mathcal{E}^d$ but the passage times $T(y,x)$ are only defined if $y \leq x$. For $x \geq 0$ in $\mathbb{Z}^d$, we define the collection $\{X_{m,n}\}$ as $X_{m,n} = T(mx,nx)$ and we see as before that if $\mathbb{E}T(0,nx) < \infty$ then $\{X_{m,n}\}$ satisfies assumptions 1-3 from Section~\ref{sec: setting}.

Now, however, the variables $X_{0,n}, X_{n,2n}, \ldots$ are independent for each $n$. So if we assume that $\underline{\chi}_2 = \overline{\chi}_{2+\delta} < \infty$ for some $\delta>0$ then association condition II holds and we can directly apply Theorem~\ref{thm:main1} to conclude $\underline{\gamma} \geq \chi$ when $\chi < 1/2$ and $\overline{\gamma} \geq \chi - \frac{1}{2} (2^{\beta^{-1}}-1)$ when $\Var T(0,nx) = O(n/(\log n)^\beta)$.

\subsection{Last-passage percolation}

In this model, as in directed FPP, paths are constrained to have non-decreasing coordinates. The weights $(t_v)$ are placed on vertices $v$ instead of edges and the passage time between two points $u \leq v$ is given by $T(u,v) = \max_{\pi:u\rightarrow v} T(\pi)$, where $\pi = \{u=x_0, \ldots, x_n=v\}$ is a directed path and $T(\pi) = \sum_{i=0}^{n-1} t_{x_i}$. (Note that we omit the last vertex in the sum.)

For a given $x \geq 0$ define the process $\{Y_{m,n}\}$ by $Y_{m,n} = T(mx,nx)$ as before, but note that $\{Y_{m,n}\}$ is super-additive. So setting $X_{m,n} = - Y_{m,n}$, we get a process that satisfies assumptions 1-3 as long as $\mathbb{E}t_v^d (\log t_v)_+^{d+\epsilon} < \infty$ for some $\epsilon>0$, (see \cite[Theorem 1]{CoxKesten}).  By independence of $X_{0,n}, X_{n,2n}, \ldots$, association condition II also holds. Therefore we can apply Theorem~\ref{thm:main1} again to $X_{m,n}$.

Translating back to the variables $T(mx,nx)$, if $g(x) = \sup_{n \in \mathbb{N}} \frac{1}{n} \mathbb{E}T(0,nx)$ and
\[
\underline{\gamma} = \liminf_n \frac{\log(ng(x) - \mathbb{E}T(0,nx))}{\log n},~ \overline{\gamma} = \limsup_n \frac{\log(ng(x)-\mathbb{E}T(0,nx))}{\log n}\ ,
\]
we find $\underline{\gamma} \geq \chi$ when $\chi < 1/2$ and $\overline{\gamma} \geq \chi - \frac{1}{2} (2^{\beta^{-1}}-1)$ when $\Var T(0,nx) = O(n/(\log n)^\beta)$.

%The exponents $\overline{\chi}_p$  and $\underline{\chi}_p$ are defined as in Definition \ref{def:exponentxi}. We let the exponents $\overline{\gamma}$ and $\underline{\gamma}$ be defined as in Definition \ref{def:exponentgamma} but with $\leq$ replaced by $\geq$. It is not difficult to see that the proof of Theorems \ref{thm:main1} and \ref{thm1.2} run exactly the same with 
%$$ g(\mathbf e) = \sup_{n \in \N} \frac{1}{n} \E T(0,n \mathbf e)$$ and subadditivity replaced by superadditivity.

The importance of this small variation of FPP is that there are correspondences between LPP models and certain queueing networks, namely systems of queues in tandem. In dimension two this connection reaches a deeper level as precise scaling laws have been obtained in  special cases. If the passage times are exponentially distributed with mean $1$ then Rost \cite[Theorem~1]{Rost} showed
$$
g(x) = g(x_1,x_2) = (\sqrt{x_1} + \sqrt{x_2})^2 \ .$$
If $t_x$ is geometric with parameter $p$ then \cite[Theorem~1.1]{Joha}
$$g(x_1,x_1) = \frac{1}{p} (x_1+x_2+ 2\sqrt{x_1x_2(1-p)})\ .$$
In both cases finer asymptotics are available \cite[Theorem~1.2]{Joha} as the distribution of
 \begin{equation}\label{Joha}
 \frac {T(0,n(x,y)) - ng(x,y)}{n^{1/3}} \Rightarrow Z \text{ as } n \to \infty
\end{equation} 
for a non-degenerate $Z$. The proof of \eqref{Joha} goes through the following special identity that identifies the law of the passage time with the law of the largest eigenvalue of the Laguerre Unitary ensemble. Let $A$ be an $n\times n$ matrix with entries that are complex Gaussian random variables with mean zero and variance $1/2$.

\begin{theorem}\label{thm:joha}\cite[Proposition 2.4]{Joha} If $t_x$ is exponentially distributed with mean $1$ then
\begin{equation}
\Pro( T(0,(n,n)) \leq t ) = \Pro (\lambda_n \leq t) \text{ for all } t \geq 0\ ,
\end{equation}
where $\lambda_n$ is the largest eigenvalue of the $n \times n$ matrix $AA^*$.
\end{theorem}
 
The law of $\lambda_n$ is explicit and amenable to asymptotic analysis through the study of Laguerre orthogonal polynomials (see \cite{Deift00} and the references therein). 
In particular, \eqref{Joha} is a combination of the theorem above and  the fact that \cite[Remark~1.5]{Joha}

\begin{equation}\label{eq:tracy}
Z_n:=  \frac{\lambda_n - 4n}{2^{4/3}n^{1/3}} \stackrel{\mathcal D}{\rightarrow} W_2 \sim F_2,\
 \end{equation}
where 
$$F_2(s) = \exp\bigg(-\int_{s}^\infty (x-s) q(x)^2 d x\bigg)$$ 
with $q$ the solution of the Painlev\'e II differential equation
\begin{eqnarray*}
q''(x)= xq(x) +2q^3(x)\\
q(x)\sim Ai(x) \text{ as } x \rightarrow +\infty
\end{eqnarray*}
and $Ai(x)$ denotes the Airy function.
Our assumptions on Theorem~\ref{thm:main1} now translate to first order asymptotics of the moments of order $k$ of $\lambda_n$. Although the limit \eqref{eq:tracy} is widely known, these asymptotics were obtained only in \cite[Corollary 1.3]{Jinho} where it was shown that 
\begin{equation}\label{eq:convmoments}
 \E Z_n^k \rightarrow \E W_2^k <\infty \text{ for any }k \in \mathbb{N}\ .
\end{equation}
In particular this implies

\begin{proposition} \label{propLPP} The assumptions of Theorems \ref{thm:main1} hold for LPP with exponential and geometric weights. Furthermore, for any $p\geq 1$, $\underline{\gamma} = \overline{\chi}_p = \underline{\chi}_p=1/3$.
\end{proposition}

\begin{proof}
The proof follows directly from \eqref{eq:convmoments} and Theorem \ref{thm:joha}. Indeed, for any $\gamma'<1/3<\gamma''$, we have for $n$ sufficiently large, $ n^{\gamma'} \leq \mathbb{E}T(0,(n,n))- n g(1,1)  \leq n^{\gamma''}$. This establishes $\underline{\gamma} = 1/3$.
In the same way, for any $\chi''<1/3<\chi'$ and for $q=2,4$, \eqref{eq:convmoments} leads to  
\[
\lim_{n\rightarrow \infty} \frac{\|T(0,(n,n))-\mathbb{E}T(0,(n,n))\|_q}{n^{\chi'}} = 0 \quad
\lim_{n\rightarrow \infty} \frac{\|T(0,(n,n))-\mathbb{E}T(0,(n,n))\|_q}{n^{\chi''}} = \infty,
\]
establishing that $\underline{\chi}_2 = \overline{\chi}_4$.
\end{proof}

\begin{remark}
In both solvable cases, one can prove by direct asymptotic analysis that $\gamma = \chi = 1/3$. The importance of Proposition~\ref{propLPP} is that the assumption $\underline{\chi}_2 = \overline{\chi}_{2+\delta}$ is indeed valid. Proving $\chi=1/3$ for general distributions is an open problem.
\end{remark}

\subsection{Bin packing}
We consider $n$ objects with random sizes $X_1, X_2, \ldots, X_n$ having a common distribution on $[0,1]$, and an unlimited collection of bins, each of size $1$. For $m<n$, let $T_{m,n}:=T_{m,n}(X_{m+1},\ldots,X_n)$ be the minimum number of bins required to pack the objects $X_{m+1}, \ldots, X_n$. Then $T_n:=T_{0,n} \leq T_{0,m} + T_{m,n} $ is sub-additive and therefore $\lim_{n}\frac{T_n}{n} = \lim_n \frac{\E T_n}{n} = g.$ Because assumptions 1-4 hold (again 4 holds by independence), Theorem~\ref{thm:main1} applies.

\subsection{Directed polymers in random environment}
As in LPP, we consider the collection of directed paths with i.i.d. nonnegative weights on the vertices and define $T(\pi)$, the passage time of a directed path $\pi$, as in that context. Given $\beta>0$ we define the partition function from $u$ to $v$ at inverse temperature $\beta$ as   
 \begin{equation*}
  Z^{\beta}(u,v) = \sum_{\pi: u\to v} \exp(- \beta T(\pi))\ ,
  \end{equation*}
where the sum runs over all directed paths from $u$ to $v$. We set 
\begin{equation}\label{eq:freeenergy}
F(u,v) = -\frac{1}{\beta} \log \frac{Z^\beta (u,v)}{d^{\|v-u\|_1}}\ .
 \end{equation}
For a given $x \geq 0$ in $\mathbb{Z}^d$, the collection $\{X_{m,n}\}$ defined by $X_{m,n} = F(mx,nx)$ is sub-additive and satisfies assumptions 1-3. Because $X_{0,n},X_{n,2n}, \ldots$ is independent, the process satisfies association condition II and again we can apply Theorem~\ref{thm:main1}.

\subsection{Longest common subsequence} 

Consider $(X_i)_{i \in \mathbb N}$ and $(Y_i)_{i \in \mathbb N}$, two sequences of i.i.d. random variables taking values in a finite alphabet $\mathcal A = \{a_1, \ldots, a_r\}$. Let $LC_n$ be the length of the longest common subsequence of $X_1, \ldots, X_n$ and $Y_1, \ldots, Y_n$; that is, $LC_n$ is the largest $k$ such that there exists $1\leq i_1 < i_2 <\ldots < i_k\leq n$ and
$1\leq j_1 < j_2 <\ldots < j_k\leq n$ with $X_{i_l} = Y_{i_l}$, $l=1,\ldots, k$.

The study of the asymptotics of $LC_n$ has a long history starting with the pioneering work of Chv\'atal and Sankoff \cite{CSankoff} where it was shown that 

$$  \frac{ \E LC_n}{n} \rightarrow g_r\ .$$
The rate of convergence of the above sequence was first investigated by Alexander \cite{AlexanderLCC}, who proved the bound $\E LC_n \geq ng_r + C\sqrt{n \log n}$ for some $C>0$. 
When $X_1$ and $Y_1$ are both Bernoulli with parameter $p=1/2$, it is conjectured \cite{CSankoff} that $\Var LC_n=o(n^{2/3})$. On the other hand, when $p$ is small enough, it is known that there exist positive constants $c,C$ such that $ cn\leq \Var LC_n \leq Cn$ \cite{Henri, Steele}. Letting $X_{m,n}$ be the longest common subsequence of $X_{m+1},\ldots, X_n$ and $Y_{m+1},\ldots, Y_n$, it is straight-forward to check that $X_{m,n}$ satisfies 1-3. Again, as $X_{0,n},X_{n,2n}, \ldots$ is independent, the process also satisfies association condition II and we can apply Theorem~\ref{thm:main1}.

\subsection{First birth problem or the minimum of a branching random walk}

Let $(t_i)_{i \in \N}$ be a sequence of non-negative i.i.d. random variables. Consider a branching process where each individual $i$ lives for a certain amount of time $t_i$. The process starts with one individual at time $0$. At the time of its death, the individual produces $k$ offspring with probability $p_k$. After that, all offspring start independent copies of the original process.
 
Assume that the branching is supercritical: $\sum_k k p_k>1$. Let $B_n$ be the birth time of the first member of generation $n$ (with $B_0=0$). $B_n$ can also be interpreted as the minimum of a branching random walk where the step sizes are given by the collection $(t_i)_{i\in \N}$. To estimate $B_n$, let $B_{0,m}$ be the birth time of the first individual in generation $m$, and $B_{m,n}$ be the time needed for this individual to have an offspring in generation $n>m$. This process was initially investigated in \cite{BB1, BB2, Kingman1}. One has $B_n = B_{0,n} \leq B_{0,m} + B_{m,n}$, and therefore there exists a constant $g$ such that $\frac{B_n}{n} \rightarrow g$ a.s. The value of $g$ can be explicitly computed if one has finite exponential moments for the offspring distribution \cite{BB1}. 
It is known that for a wide range of branching random walks \cite{Adario} that $\E B_n - ng  = c \log n + O(1)$ for some constant $c$; thus $\underline{\gamma} = \overline{\gamma} = 0$. More information is available if the offspring distribution is in the boundary case (see \cite[Equation~(1.1)]{Aidekon} for a definition), where $g =0$ and $B_n - c\log n$ converges in distribution. 
 
Again, assumptions 1-3 hold, and since $B_{0,n}, B_{n,2n}, \ldots$ are independent for each $n$, we can apply Theorem~\ref{thm:main1}.

\section{Sketch of the proof of Theorem \ref{thm:main1}} \label{sec:sketch}

Because the proofs of the main theorems are somewhat technical, we will give here a sketch of the ideas.

We argue by contradiction, so assume first that $\chi < 1/2$ but $ \gamma < \chi$. For a random variable $X$, let $X' = X - \E X$. 

{\bf Step 1:} (Central limit theorem.) The first step is to use weak dependence to show that for a sequence $l_n \to \infty$ and $\epsilon>0$ 

$$\frac{ X_{0,l_n}' + \cdots + X_{(k_n-1)l_n,k_nl_n}' }{\sqrt{k_n \Var X_{0,l_n}}} \Longrightarrow N(0,1), $$ 
where $k_n \sim l_n^{\epsilon}$. This is the goal of Proposition \ref{prop: main_estimate}. (If the sequence $X_{0,n}, X_{n,2n},X_{2n,3n} \ldots$ is independent and identically distributed then this convergence follows by a routine application of Lyapunov's condition). Note that this convergence does not imply a central limit theorem for the sequence $(X_{0,n})$, as we are summing shifted copies above. However, it will imply a lower bound for the lower tail fluctuations of $X_{0,n}$ in the next step.

{\bf Step 2:} We define $$v(n) = \frac{\Var X_{0,n}}{n}.$$ In this part of the proof, we use the first step to derive a recursive inequality for $v(n)$ that leads to 
\begin{equation}\label{eq:sketch0}
\liminf_n \frac{v(k_nl_n)}{v(l_n)} \geq 1/2.
\end{equation}

This can be roughly justified as follows. We start by using sub-additivity and the fact that $\E X_{0,mn}~\geq~mng$  to obtain the inequality 
\begin{align}\label{eq:sketch3}
X'_{0,mn} &\leq \sum_{k=1}^m X'_{(k-1)n,kn} + m \mathbb{E}X_{0,n}- \mathbb{E}X_{0,mn} \nonumber \\
&\leq \sum_{k=1}^m X'_{(k-1)n,kn} + m\left[ \mathbb{E}X_{0,n}-ng\right]\ .
\end{align}
If we square both sides of the above inequality when the right term is nonpositive we obtain

\begin{equation}\label{eq:sketch1}
\Var X_{0,mn} \geq \E \bigg(\sum_{k=1}^m X'_{(k-1)n,kn} + m \left[ \mathbb{E}X_{0,n}-ng\right]  \bigg)_{-}^2. 
\end{equation}

Set  $m \sim n^{\epsilon}$ (or more precisely replace $m$ by $k_n$ and $n$ by $l_n$) in Step 1. Then, the first term is of order $\sqrt{m \Var{X_{0,n}}} \sim m^{1/2}n^{\chi}$  while the second term is at most $mn^{\gamma}$, which is smaller because

$$n^{\epsilon/2 + \chi} \gg n^{\epsilon + \gamma} \quad \text{ for } \epsilon \text{ small. }$$ 
Thus, dividing \eqref{eq:sketch1} by $mn$ we get 
$$\frac{\Var X_{0,mn}}{mn} \gtrsim \frac{1}{n} \E \bigg( \frac{1}{\sqrt{m}} \sum_{k=1}^m X_{(k-1)n,kn}'\bigg)^2_{-} \gtrsim \frac{1}{2} \frac{\Var X_{0,n}}{n},$$
which gives \eqref{eq:sketch0}. This step is done in Corollary \ref{corollary: main_cor}.

{\bf Step 3:} One can show (see Section \ref{sec: main_thms}) that Equation \eqref{eq:sketch0} implies $\chi \geq 1/2$, which is a contradiction as we earlier assumed that $\chi < 1/2$. In the case $\chi >1/2$, we use a different inequality to replace Step 2 to show that if $\gamma < \chi$ then $\chi \leq 1/2$. 

\section{Proofs}\label{sec:proof}

We start with the simplest of our proofs, the proof of Theorem~\ref{thm: shorter}.
\subsection{Proof of Theorem~\ref{thm: shorter}}\label{sec:shorter}
Recall the notation $X_{0,n}' = X_{0,n} - \mathbb{E}X_{0,n}$.  
The main bound we will use here, for $m,n \geq 1$, is \eqref{eq:sketch3} which we rewrite for convenience of the reader:
\begin{align*}
X'_{0,mn} &\leq \sum_{k=1}^m X'_{(k-1)n,kn} + m \mathbb{E}X_{0,n}- \mathbb{E}X_{0,mn} \\
&\leq \sum_{k=1}^m X'_{(k-1)n,kn} + m\left[ \mathbb{E}X_{0,n}-ng\right]\ .
\end{align*}
From this inequality we obtain two others. Set $X_+ = X \mathbf{1}_{\{X \geq 0\}}$ and $X_- = X-X_+$.

The first inequality will help when $\chi>1/2$:
\begin{equation}\label{eq: upper_bound}
\mathbb{E}(X_{0,mn}')_+^2 \leq  \Var \sum_{k=1}^m X_{(k-1)n,kn} + m^2 \left[ \mathbb{E}X_{0,n}-ng \right]^2\ . 
\end{equation}
We will use the second when $\chi<1/2$:
\begin{equation}\label{eq: lower_bound}
\Var X_{0,mn} \geq \mathbb{E}\left( \sum_{k=1}^m X'_{(k-1)n,kn} + m\left[ \mathbb{E}X_{0,n}-ng\right] \right)_-^2\ .
\end{equation}

We will also make liberal use of the following variant of the Paley-Zygmund inequality.
\begin{lem}\label{lem: PZ}
If $X$ has mean zero with $\|X\|_2\in (0,\infty)$ then for $\theta \in (0,1)$ and $\delta>0$,
\[
\mathbb{P}\left( |X| \geq \theta \sqrt{\Var X} \right) \geq (1-\theta^2)^{1+2/\delta} \left( \frac{\|X\|_2}{\|X\|_{2+\delta}} \right)^{2+4/\delta}\ .
\]
\end{lem}

\begin{proof}
We may assume that $\|X\|_{2+\delta} < \infty$. If $X$ if a nonnegative random variable with $\|X\|_p<\infty$ and $\theta \in (0,1)$,
\[
\mathbb{E}X = \mathbb{E}X\mathbf{1}_{\{X < \theta \mathbb{E}X\}} + \mathbb{E}X\mathbf{1}_{\{X \geq \theta \mathbb{E}X\}} \leq \theta\mathbb{E}X + \|X\|_p \left(\mathbb{P}(X \geq \theta \mathbb{E}X)\right)^{1/q}\ ,
\]
where $p,q > 1$ satisfy $p^{-1}+q^{-1}=1$. For $\|X\|_p>0$, we obtain
\[
\mathbb{P}(X \geq \theta \mathbb{E}X) \geq \left( (1-\theta) \frac{\mathbb{E}X}{\|X\|_p} \right)^q\ .
\] 
Put $2p=2+\delta$, so that 
\[
q = \frac{p}{p-1} = \frac{1+\delta/2}{\delta/2} = 1+ \frac{2}{\delta}\ .
\]
Replace by $X^2$:
\[
\mathbb{P}(|X| \geq \theta \|X\|_2) \geq \left( (1-\theta^2) \frac{\mathbb{E}X^2}{\|X^2\|_p} \right)^q = (1-\theta^2)^{1+2/\delta} \left( \frac{\|X\|_2}{\|X\|_{2+\delta}} \right)^{2+4/\delta}\ .
\]
\end{proof}

\subsubsection{The case $\chi>1/2$}
Assume 1-3, (M1) and $\mathbb{E}|X_{0,n}|^{2+\delta}<\infty$ for some $\delta>0$ and all $n$. Inequality \eqref{eq: upper_bound} then implies
\[
\mathbb{E}(X_{0,mn}')_+^2 \leq \CC_1 m \Var X_{0,n} + m^2 \left[ \mathbb{E}X_{0,n} - ng \right]^2\ .
\]
Because 
\[
0 = \mathbb{E}X_{0,mn}' = \mathbb{E}(X_{0,mn}')_+ + \mathbb{E}(X_{0,mn}')_-\ ,
\]
our inequality becomes
\begin{align}
\mathbb{E}|X_{0,mn}'| &= 2\mathbb{E}(X_{0,mn}')_+ \nonumber \\
&\leq \sqrt{4\CC_1 m \Var X_{0,n} +4m^2 \left[ \mathbb{E}X_{0,n} - ng \right]^2} \label{eq: cheez_wiz}\ .
\end{align}

Apply Lemma~\ref{lem: PZ} to $X = X_{0,mn}'$ with $\theta=1/2$ and combine with \eqref{eq: cheez_wiz}:
\begin{align}
\frac{1}{2} \cdot \left( \frac{3}{4} \right)^{1+2/\delta} \left( \frac{\|X_{0,mn}'\|_2}{\|X_{0,mn}'\|_{2+\delta}} \right)^{2+4/\delta} \|X_{0,mn}'\|_2 &\leq \mathbb{E}|X_{0,mn}'| \nonumber \\
&\leq \sqrt{ 4\CC_1m\Var X_{0,n} + 4m^2\left[ \mathbb{E}X_{0,n}-ng \right]^2} \label{eq: cheez_wiz_2}\ ,
\end{align}
so long as $\Var X_{0,mn} > 0$.

Now assume that $\chi:= \underline{\chi}_2 = \overline{\chi}_{2+\delta} \in (1/2,\infty)$ but that $\underline{\gamma} < \chi$. Let $l_n \to \infty$ be such that
\[
\frac{\log (\mathbb{E}X_{0,l_n} - l_ng)}{\log l_n} \to \underline{\gamma}
\]
and for each $n$, let $k_n \geq 1$. Considering the left side of \eqref{eq: cheez_wiz_2}:
\[
 \lim_n\frac{\log \left[ \frac{1}{2} \left( \frac{3}{4}\right)^{1+2/\delta} \left( \frac{\|X'_{0,k_nl_n}\|_2}{\|X_{0,k_nl_n}'\|_{2+\delta}}\right)^{2+4/\delta} \|X_{0,k_nl_n}'\|_2\right]}{\log k_nl_n} = \chi\ .
\]
Using the right side of \eqref{eq: cheez_wiz_2}, then,
\begin{align*}
2\chi &\leq \liminf_n \frac{\log \left( 4\CC_1k_n \Var X_{0,l_n} + 4k_n^2 \left[\mathbb{E}X_{0,l_n} - l_ng\right]^2 \right)}{\log k_nl_n} \\
&\leq \liminf_n \max \left\{ \frac{\log k_n \Var X_{0,l_n}}{\log k_nl_n}, \frac{\log k_n^2 \left[\mathbb{E}X_{0,l_n}-l_ng\right]^2}{\log k_nl_n} \right\}\ .
\end{align*}
Choose $k_n = \lfloor l_n^\epsilon \rfloor$ for $\epsilon>0$. Then
\begin{align*}
\frac{\log k_n \Var X_{0,l_n}}{\log k_nl_n} &= \frac{\log k_n}{\log k_n+\log l_n} + 2\frac{\log l_n}{\log k_n+\log l_n} \cdot \frac{\log \|X_{0,l_n}'\|_2}{\log l_n} \\
&\to \frac{\epsilon}{1+\epsilon} + 2\chi \frac{1}{1+\epsilon}
\end{align*}
and
\[
\frac{\log k_n^2 \left[ \mathbb{E}X_{0,l_n}-l_ng\right]^2}{\log k_nl_n} \to 2\frac{\epsilon}{1+\epsilon} +2\underline{\gamma} \frac{1}{1+\epsilon}\ .
\]
Therefore
\begin{equation}\label{eq: pizzeria}
2\chi \leq \max\left\{ \frac{\epsilon}{1+\epsilon} + 2\chi\frac{1}{1+\epsilon}, 2\frac{\epsilon}{1+\epsilon} + 2\underline{\gamma} \frac{1}{1+\epsilon} \right\}\ .
\end{equation}
or
\[
2\chi(1+\epsilon) \leq \max\{\epsilon+2\chi, 2\epsilon + 2\underline{\gamma}\}\ .
\]
For $\epsilon$ small, the dominant term on the right is the first, so $\chi \leq 1/2$, a contradiction. Therefore $\underline{\gamma} \geq \chi$.

%Next assume that $\chi:= \overline{\chi}_2 = \overline{\chi}_{2+\delta} \in (1/2,\infty)$ but $\overline{\gamma} < \chi$. Given $\epsilon>0$, choose $l_n \to \infty$ such that, if we set $k_n = \lfloor l_n^\epsilon \rfloor$,
%\[
% \frac{\log \Var X_{0,k_nl_n}}{\log k_nl_n} \to 2\chi\ .
%\]
%Then, as above,
%\[
%2\chi \leq \liminf_n \max \left\{ \frac{\log k_n \Var X_{0,l_n}}{\log k_nl_n}, \frac{\log k_n^2 \left[\mathbb{E}X_{0,l_n}-l_ng\right]^2}{\log k_nl_n} \right\}\ ,
%\]
%giving \eqref{eq: pizzeria}, but with $\overline{\gamma}$ replacing $\underline{\gamma}$. Argue using $\overline{\gamma} < \chi$ to find $\chi\leq 1/2$, a contradiction, so $\overline{\gamma} \geq \chi$.

\subsubsection{The case $\chi<1/2$}
Assume 1-3, (M2), and $\|X_{0,n}\|_{2+\delta}<\infty$ for some $\delta>0$ and all $n$. We begin with a simple inequality: If $a \in \mathbb{R}$ and $b \geq 0$ then
\begin{align*}
(a+b)_-^2 = (a+b)^2 \mathbf{1}_{\{a \leq -b\}} \geq (a+b)^2\mathbf{1}_{\{a \leq -2b\}} &\geq \frac{1}{4} a^2 \mathbf{1}_{\{a \leq -2b\}} \\
&= \frac{1}{4} \left[ a^2 \mathbf{1}_{\{a \leq 0\}} - a^2 \mathbf{1}_{\{a \in (-2b,0]\}}\right] \\
&\geq \frac{1}{4}a_-^2 - b^2\ .
\end{align*}
Use this in inequality \eqref{eq: lower_bound} for
\begin{align*}
\Var X_{0,mn} &\geq \frac{1}{4} \mathbb{E}\left(\sum_{k=1}^m X'_{(k-1)n,kn} \right)_-^2 - m^2 \left[ \mathbb{E}X_{0,n}-ng \right]^2 \\
&\geq \frac{1}{16} \left( \mathbb{E}\left|\sum_{k=1}^m X'_{(k-1)n,kn} \right| \right)^2 - m^2 \left[ \mathbb{E}X_{0,n}-ng \right]^2\ .
\end{align*}
By Lemma~\ref{lem: PZ}, if we write 
\[
S_{m,n} = \sum_{k=1}^{m} X'_{(k-1)n,kn}\ ,
\]
then as long as $\Var S_{m,n} > 0$,
\begin{equation}\label{eq: other}
\Var X_{0,mn} \geq \frac{1}{64} \left(\frac{3}{4} \right)^{2+4/\delta} \left( \frac{\|S_{m,n}\|_2}{\|S_{m,n}\|_{2+\delta}}\right)^{4+8/\delta} \|S_{m,n}\|_2^2 - m^2 \left[ \mathbb{E}X_{0,n} - ng\right]^2\ .
\end{equation}

Now assume that $\chi:= \underline{\chi}_2 = \overline{\chi}_{2+\delta} <1/2$ but that $\underline{\gamma} < \chi$. Choose sequences $(l_n)$ and $(k_n)$ such that $k_n \geq 1$ and 
\[
\frac{\log(\mathbb{E}X_{0,l_n} - l_ng)}{\log l_n} \to \underline{\gamma}\ .
\]
Choose $m=k_n$ and $n=l_n$ in \eqref{eq: other}, and use (M2):
\[
\Var X_{0,l_nk_n} + k_n^2 \left[ \mathbb{E}X_{0,l_n} - l_ng \right]^2 \geq \frac{\CC_2}{64} \left( \frac{3}{4} \right)^{2+4/\delta} \left( \frac{(\CC_2)^{1/2} \|X_{0,l_n}'\|_2}{(\CC_3)^{\frac{1}{2+\delta}}\|X_{0,l_n}'\|_{2+\delta}} \right)^{4+8/\delta} k_n \Var X_{0,l_n}\ .
\]
Note that if $k_n = \lfloor l_n^\epsilon \rfloor$ for some $\epsilon>0$, then
\[
\frac{\log \Var X_{0,l_nk_n}}{\log l_nk_n} \to 2\chi ,
\]
\[
\frac{\log k_n^2 \left[ \mathbb{E}X_{0,l_n}-l_ng \right]^2}{\log l_nk_n} \to 2\frac{\epsilon}{1+\epsilon} + 2\underline{\gamma}\frac{1}{1+\epsilon}
\]
and
\[
\frac{\log\left(  \frac{\CC_2}{64} \left( \frac{3}{4} \right)^{2+4/\delta} \left( \frac{(\CC_2)^{1/2} \|X_{0,l_n}'\|_2}{(\CC_3)^{\frac{1}{2+\delta}}\|X_{0,l_n}'\|_{2+\delta}} \right)^{4+8/\delta} k_n \Var X_{0,l_n}\right)}{\log l_nk_n} \to\frac{\epsilon}{1+\epsilon} + 2\chi \frac{1}{1+\epsilon}\ .
\]
So we obtain
\begin{equation}\label{eq: pizzeria2}
\frac{\epsilon}{1+\epsilon} + 2\chi \frac{1}{1+\epsilon} \leq \max\left\{ 2\chi, 2\frac{\epsilon}{1+\epsilon} + 2\underline{\gamma}\frac{1}{1+\epsilon} \right\}\ .
\end{equation}
If $\epsilon$ is small, the dominant term on the right is $2\chi$, so for such $\epsilon$,
\[
\epsilon + 2\chi \leq 2\chi (1+\epsilon)\ .
\]
This means $\chi \geq 1/2$, a contradiction.

Last assume that $\chi:= \overline{\chi}_2 = \overline{\chi}_{2+\delta}<1/2$ for some $\delta>0$ but that $\overline{\gamma} < \chi$. Pick $l_n \to \infty$ such that
\[
 \frac{\log \Var X_{0,l_n}}{\log l_n} \to 2\chi
\]
and, given $\epsilon>0$, set $k_n = \lfloor l_n^\epsilon \rfloor$. Arguing as above, we obtain \eqref{eq: pizzeria2}, but with $\overline{\gamma}$ replacing $\underline{\gamma}$. Using $\overline{\gamma} < \chi$, we find $\chi \geq 1/2$, a contradiction, so $\overline{\gamma} \geq \chi$.

%
%In this section, we use central limit theorems to bound $\Var X_{0,kn}$ in terms of $\Var X_{0,n}$ for various values of $k,n \geq 1$ (see Corollary~\ref{corollary: main_cor}). We will use this result for the proof of Theorem~\ref{thm:main1} as follows. If we assume, for a contradiction, that $\gamma < \chi < 1/2$ (where $\gamma = \overline{\gamma}$ or $\underline{\gamma}$, depending on the context), then $\mathbb{E}X_{0,n} - ng$ will be sufficiently small compared to $n^\chi$ and we can apply this bound to show $\Var X_{0,n} \gg n^{2\chi}$. This strategy is inspired by Kesten's proof of \cite[Eq.~(1.16)]{Kesten} (see also a similar idea in \cite{Chatterjee2013}).

\subsection{Application of central limit theorems}

Our first goal  in this section is to derive a bound on the fluctuations of sums of terms $X_{(j-1)l_n,jl_n}$ for suitably chosen $l_n$. Under our asymptotic independence assumptions, we can apply a central limit theorem. For the statement below, denote by $\Phi$ the distribution function of a standard normal variable, and recall the definition of $\CC_5$ from item 2 of mixing condition I.

\begin{prop}\label{prop: main_estimate}
Given $\delta>0$, there exists $\Delta = \Delta(\delta)>0$ such that for any sequences of positive integers $(l_n)$, $(k_n)$ with $l_n \to \infty$, $\|X_{0,l_n}'\|_2>0$, and
\begin{equation}\label{eq: k_n_condition}
\frac{\|X_{0,l_n}'\|_{2+\delta}}{\|X_{0,l_n}'\|_2} = o(k_n^{\Delta})\ ,
\end{equation}
one has, for all $y \leq 0$,
\[
\liminf_n \mathbb{P}\left( X_{0,l_n}' + \cdots + X_{(k_n-1)l_n,k_nl_n}' \leq y \|X_{0,l_n}'\|_2\sqrt{k_n}\right) \geq C(y)\ .
\]
Under mixing condition I, $C(y) = \Phi(y/\sqrt{\CC_5})$ and under association condition II, $C(y) = \Phi(y)$. Last, under association condition II one may take $\Delta = \frac{1}{2} - \frac{1}{2+\delta}$.
\end{prop}

\begin{proof}
We will apply central limit theorems for weakly dependent variables. We will handle assumption 4 differently depending on whether or not condition I holds. So take sequences $(l_n)$ and $(k_n)$ as in the statement of the proposition. This implies, in particular, that by removing finitely many terms,
\begin{equation}\label{eq: moment_properties}
\|X_{0,l_n}'\|_{2+\delta} <\infty \text{ and } \|X_{0,l_n}'\|_2 >0 \text{ for all }n
\end{equation}
and, for any given $\Delta>0$, by Jensen's inequality,
\begin{equation}\label{eq: k_n}
k_n \to \infty \text{ as } n \to \infty\ .
\end{equation}

\medskip
\noindent
{\bf Under condition I.} In this case, we will use a central limit theorem for triangular arrays under strong mixing. We were unable to find such a result in the literature, so we give a proof in Appendix~\ref{sec: mixing_clt}.
\begin{lem}\label{lem: alpha_CLT}
Let $\{\eta_i^{(n)} : i,n \geq 1\}$ be an array of random variables such that for each $n$, $(\eta_i^{(n)})$ is a stationary sequence. Assume that there exist $\CC_7, \CC_8$ and $\delta>0$ such that, for each $n$,
\begin{enumerate}
\item $\mathbb{E}\eta_1^{(n)}=0$ and $\mathbb{E}\left( \eta_1^{(n)} \right)^2 = 1$;
\item $\CC_7 b \leq \mathbb{E} \left| \sum_{i=1}^b \eta_i^{(n)} \right|^2$ for all $b$;
\item $\mathbb{E}\left| \sum_{i \in \mathcal{B}} \eta_i^{(n)} \right|^2 \leq \CC_8 b$ for all $b$ and sets $\mathcal{B}$ of indices such that $\#\mathcal{B} = b$;
\item $\mathbb{E}|\eta_1^{(n)}|^{2+\delta} < \infty$;
\item for each $\kappa>0$, there exists $\CC_9$ such that, for all $n$ and $x$, $\alpha_n(x) \leq \CC_9x^{-\kappa}$.
\end{enumerate}
Then, letting $\sigma(n)^2 = \Var \sum_{i=1}^{r(n)} \eta_i^{(n)}$, there exists $\Delta>0$ such that
\begin{equation}\label{eq:cltsum}
\frac{1}{\sigma(n)} \sum_{i=1}^{r(n)} \eta_i^{(n)} \Rightarrow N(0,1)
\end{equation}
for any increasing $r(\cdot)$ such that $\|\eta_1^{(n)}\|_{2+\delta}/r(n)^\Delta \to 0$ as $n \to \infty$.
\end{lem}
In the lemma, the mixing coefficient $\alpha_n$ is the one associated to the $n$-th row of the array $(\eta_i^{(n)})$. That is, 
\[
\alpha_n(x) = \sup_{k \geq 1} \alpha\left( \sigma(\eta_1^{(n)}, \ldots, \eta_k^{(n)}), \sigma(\eta_{k+x}^{(n)}, \ldots) \right)\ ,
\]
where for two sigma-algebras $\Sigma_1$ and $\Sigma_2$, $\alpha(\Sigma_1,\Sigma_2)$ is defined as $\sup_{A \in \Sigma_1, B \in \Sigma_2} |\mathbb{P}(A \cap B)-\mathbb{P}(A)\mathbb{P}(B)|$.

We apply the lemma using $\eta_i^{(n)} = \frac{X_{(i-1)l_n, il_n}'}{\|X_{0,l_n}'\|_2}$. Then conditions 1 and 4 hold by \eqref{eq: moment_properties}, and 2, 3 and 5 hold by mixing condition I. Setting $r(n) = k_n$, then
\[
\|\eta_1^{(n)}\|_{2+\delta}/r(n)^\Delta = \frac{\|X_{0,l_n}'\|_{2+\delta}}{\|X_{0,l_n}'\|_2 k_n^\Delta}\ .
\]
Assuming this converges to 0, then the central limit theorem implies that for each $y \in \mathbb{R}$,
\[
\lim_n \mathbb{P}\left( X_{0,l_n}' + \cdots + X_{(k_n-1)l_n,k_nl_n}' \leq y\hat \sigma(n) \right) = \Phi(y)\ ,
\]
where $\hat \sigma(n)^2 = \Var (X_{0,l_n} + \cdots + X_{(k_n-1)l_n,k_nl_n})$. By part 2 of condition I, $\hat \sigma(n)^2 \geq \CC_5 k_n \|X_{0,l_n}'\|_2^2$, so if $y \leq 0$,
\[
\liminf_n \mathbb{P}\left( X_{0,l_n}' + \cdots + X_{(k_n-1)l_n,k_nl_n}' \leq y \sqrt{k_n} \|X_{0,l_n}'\|_2 \right) \geq \Phi\left( y/\sqrt{\CC_5} \right)\ .
\]

\medskip
\noindent
{\bf Under condition II.} Under positive association, we will use the following variant of the central limit theorem of Cox-Grimmett \cite[Theorem~1.2]{CG}. The proof is nearly identical to that in \cite{CG} but, for completeness, we outline it in Appendix~\ref{appendix: CG}. 

\begin{lem}\label{lem: CG}
Let $\{\eta_i^{(n)} : i,n \geq 1\}$ be an array of random variables such that for each $n$, $(\eta_i^{(n)})$ is a stationary and positively associated sequence. Assume further that for some sequence $(r(n))$ of positive integers, the following two conditions are met.
\begin{enumerate}
\item For some $\delta>0$ and every $i,n$,
\[
\mathbb{E}\eta_i^{(n)}=0,~ \|\eta_i^{(n)}\|_2=1 \text{ and } \|\eta_i^{(n)}\|_{2+\delta} < \infty\ .
\]
\item There is a function $\hat u : \{0, 1, 2, \ldots\} \to \mathbb{R}$ such that $\hat u(r) \to 0$ as $r \to \infty$ and for all $n \geq 1$, $i=1, \ldots, r(n)$ and $r \geq 0$,
\[
\sum_{j=1, \ldots, r(n) : |i-j| \geq r} \Cov (\eta_j^{(n)},\eta_i^{(n)}) \leq \hat u(r)\ .
\]
\end{enumerate}
If
\begin{equation}\label{eq: r_n_condition}
\|\eta_1^{(n)}\|_{2+\delta} = o\left( r(n)^{\frac{1}{2} - \frac{1}{2+\delta}} \right) \text{ as } n \to \infty
\end{equation}
then, setting $\sigma(n)^2 = \Var \sum_{i=1}^{r(n)} \eta_i^{(n)}$,
\[
\frac{1}{\sigma(n)} \sum_{i=1}^{r(n)} \eta_i^{(n)} \Rightarrow N(0,1)\ .
\]
\end{lem}

Apply Lemma~\ref{lem: CG} with $\eta_i^{(n)} = \frac{X_{(i-1)l_n,il_n}'}{\|X_{0,l_n}'\|_2}$, $r(n) = k_n$ and $\hat u = u$ from association condition II. Then \eqref{eq: moment_properties} implies item 1, and item 2 follows from \eqref{eq:covcond}. Last, \eqref{eq: r_n_condition} holds because of condition \eqref{eq: k_n_condition} on $k_n$. So for each $y \in \mathbb{R}$,
\[
\lim_n \mathbb{P}\left( X_{0,l_n} + \cdots + X_{(k_n-1)l_n,k_nl_n} \leq y \sigma(n) \right) = \Phi(y)\ .
\]
Due to positive association, $\sigma(n)^2 \geq k_n \|X_{0,l_n}'\|_2^2$, so this implies the proposition.
\end{proof}

\subsection{Iterative bound}

%\begin{figure*}[h]
%\centering
%\scalebox{1}{\includegraphics[width=1\textwidth]{Picture.pdf}}
%\caption{Construction of the block decomposition in the case $M=1$. Paths of the same type (solid, dashed, double dashed) are independent of each other as they use different collection of edges. The centered passage time of each path represents is given by $X_{k,n}'$ for some $k$. In the proof of Proposition 2.5 we sum all paths of same type giving rise to the partial sums $S_j(n)'$.}
%\label{fig:figurejordan}
%\end{figure*}

As a consequence of the last section, we can state a relation between growth of the mean of our process and the variance. Let $\Delta$ be from Proposition~\ref{prop: main_estimate}.
\begin{cor}\label{corollary: main_cor}
Let $\delta>0$. For any sequences of integers $(l_n)$, $(k_n)$ with $l_n \to \infty$,
\[
\frac{\|X_{0,l_n}'\|_{2+\delta}}{\|X_{0,l_n}'\|_2} = o(k_n^\Delta)\ , 
\]
and 
\[
k_n = o\left( \frac{\Var X_{0,l_n}}{(\mathbb{E}X_{0,l_n} - l_n g)^2} \right) \ ,
\]
one has
\[
\liminf_n \frac{\Var X_{0,k_nl_n}}{k_n \Var X_{0,l_n}} \geq \Lambda^{-1}\ ,
\]
where $\Lambda$ is given above the statement of Theorem~\ref{thm:main1}.
\end{cor}
\begin{proof}
Again by removing finitely many terms, \eqref{eq: moment_properties} holds and $k_n \to \infty$. Use \eqref{eq: lower_bound} with $m=k_n$ and $n=l_n$ for
\[
\Var A_n \geq \mathbb{E}(B_n+c_n)_-^2\ ,
\]
where
\[
A_n = \frac{X_{0,k_nl_n}'}{\sqrt{k_n} \|X_{0,l_n}'\|_2},~ B_n = \frac{1}{\sqrt{k_n} \|X_{0,l_n}'\|_2} \sum_{k=1}^{k_n} X_{(k-1)l_n,kl_n}',~ c_n = \frac{\sqrt{k_n}}{\|X_{0,l_n}'\|_2} [\mathbb{E}X_{0,l_n} - l_n g]\ .
\]
By integration by parts and Fatou's lemma,

\begin{align*}
\liminf_n \Var A_n \geq \liminf_n \mathbb{E}(B_n+c_n)^2 \mathbf{1}_{\{B_n+c_n \leq 0\}} &= 2 \liminf_n \int_{-\infty}^0 |y|~\mathbb{P}(B_n+c_n \leq y)~\text{d}y \\
&\geq 2 \int_{-\infty}^0 |y| \liminf_n \mathbb{P}(B_n+c_n \leq y)~\text{d}y\ .
\end{align*}
By assumption, $c_n \to 0$. So putting $K=\CC_5^{-1/2}$under mixing condition I and $K=1$ under association condition II, we obtain the lower bound (refer to Proposition \ref{prop: main_estimate} for the definition of $C(y)$):
\[
2 \int_{-\infty}^0 |y|~C(y)~\text{d}y = 2 \int_{-\infty}^0 |y|~\Phi(Ky)~\text{d}y = \frac{2}{K^2} \int_{-\infty}^0 |y|~\Phi(y)~\text{d}y = \frac{1}{2K^2} = \Lambda^{-1}\ .
\]

\end{proof}

\subsection{Proof of Theorem~\ref{thm:main1}}\label{sec: main_thms}
We first note that if assumption 4 holds, then so does (M1). So by Theorem~\ref{thm: shorter}, if $\chi:= \underline{\chi}_2 = \overline{\chi}_{2+\delta} \in (1/2,\infty)$ for some $\delta>0$ then $\underline{\gamma} \geq \chi$.

Next assume that $\chi:= \underline{\chi}_2 = \overline{\chi}_{2+\delta} < 1/2$ for some $\delta>0$ but that $\underline{\gamma}<\chi$. Then we can find $\epsilon>0$ and a integer sequence $(l_n)$ with $l_n \to \infty$ such that
\[
\frac{\log (\mathbb{E}X_{0,l_n} - l_n g)}{\log l_n} < \chi-\epsilon \text{ for all } n\ .
\]
%In other words,
%\[
%\mathbb{E}X_{0,l_n} - l_n g < l_n^{\chi-\epsilon} \text{ for all } n\ .
%\]
By definition of $\underline{\chi}_2$, 
%\[
%\Var X_{0,l_n} \geq l_n^{2(\chi - \epsilon/2)} \text{ for all large } n\ ,
%\]
we can restrict to a subsequence to ensure
\[
\left( \mathbb{E}X_{0,l_n} - l_n g \right)^2 < l_n^{-\epsilon} \Var X_{0,l_n} \text{ for all } n\ .
\]
Setting $k_n = \left\lfloor l_n^{\epsilon/2} \right\rfloor$, we obtain
\[
k_n = o\left( \frac{\Var X_{0,l_n}}{(\mathbb{E}X_{0,l_n} - l_n g)^2} \right)\ .
\]
By the fact that $\overline{\chi}_{2+\delta} = \underline{\chi}_2$,
\[
\frac{\|X_{0,l_n}'\|_{2+\delta}}{\|X_{0,l_n}'\|_2} = o(l_n^{\epsilon \Delta/2}) = o(k_n^\Delta)\ ,
\]
where $\Delta>0$ is from Proposition~\ref{prop: main_estimate}. By Corollary~\ref{corollary: main_cor},
\[
\Var X_{0,k_n l_n} \geq (2\Lambda)^{-1} k_n \Var X_{0,l_n} \text{ for large }n\ .
\]
Taking logarithms,
\[
2 \log \|X_{0,k_n l_n}'\|_2 \geq \log ~(2\Lambda)^{-1} + \log k_n + 2\log \|X_{0,l_n}'\|_2\ ,
\]
or
\[
2 \frac{\log \|X_{0,k_n l_n}'\|_2}{\log l_nk_n} \geq \frac{\log ~(2\Lambda)^{-1}}{\log k_nl_n} + \frac{\log k_n}{\log k_nl_n} + 2\frac{\log \|X_{0,l_n}'\|_2}{\log k_nl_n}\ .
\]
Take $n \to \infty$. Note that $\frac{\log k_n}{\log k_nl_n} \to \frac{\epsilon}{2+\epsilon}$ and use the fact that $\underline{\chi}_2 = \overline{\chi}_2$ to get
\[
2\chi \geq \frac{\epsilon}{2+\epsilon} + \frac{4}{2+\epsilon}\chi\ ,
\]
or $\chi \geq 1/2$, a contradiction.

If we assume instead that $\chi:= \overline{\chi}_2 =\overline{\chi}_{2+\delta} <1/2$ for some $\delta>0$ but $\overline{\gamma} < \chi$, we repeat the above argument, but along a subsequence. That is, let $(l_n)$ be an increasing sequence with $l_n \rightarrow \infty$ along which
\[
\frac{\log \Var X_{0,l_n}}{2 \log l_n} \rightarrow \overline \chi \quad \text{as }n \rightarrow \infty\ . 
\]
We can, as above, choose $\epsilon>0$ such that, if $k_n = \lfloor l_n^{\epsilon/2}\rfloor$, then $k_n = o\left( \frac{\Var X_{0,l_n}}{(\mathbb{E}X_{0,l_n}-l_n g)^2} \right)$. Furthermore $\|X_{0,l_n}'\|_{2+\delta}/\|X_{0,l_n}'\|_2 = o(k_n^\Delta)$ and so we obtain $\chi \geq 1/2$, a contradiction.

Last, assume that for some $\delta>0$, $\underline{\chi}_2 = \overline{\chi}_{2+\delta}<\infty$, but that $\overline{\gamma} < \chi - E$ for $E:=\frac{1}{2}(\Lambda^{\beta^{-1}}-1)$ and $\Var X_{0,n} = O\left( \frac{n}{(\log n)^\beta}\right)$ for some $\beta>0$. This implies that $-\infty< \chi \leq 1/2$. Then for any $\epsilon>0$ sufficiently small, there exists $N \geq 2$ such that
\[
(\mathbb{E}X_{0,n} - ng)^2 < n^{-2E -\epsilon} \Var X_{0,n} \text{ for } n\geq N\ .
\]
Thus if we set
\[
k_n = \min \left\{ \left\lceil \frac{\Var X_{0,n}}{(\mathbb{E}X_{0,n} - ng)^2} n^{-\epsilon/2}\right\rceil, \left\lceil n^{2E + \epsilon/2} \right\rceil \right\}\ ,
\]
then $k_n \geq n^{2E+\epsilon/2}$ for $n \geq N$ and so if we choose $\Delta>0$ from Proposition~\ref{prop: main_estimate}, then
\[
k_n^\Delta \geq n^{(2E+\epsilon/2)\Delta} \text{ for } n\geq N\ .
\]
Last, from $\underline{\chi}_2 = \overline{\chi}_{2+\delta}$, 
\[
\frac{\|X_{0,n}'\|_{2+\delta}}{\|X_{0,n}'\|_2} =  o(k_n^\Delta)\ .
\]
This means we can apply Corollary~\ref{corollary: main_cor} with $l_n=n$ to find, for any positive $a < \Lambda^{-1}$, an $N' \geq N$ such that
\[
\Var X_{0,k_n n} \geq a k_n \Var X_{0,n} \text{ for } n \geq N'\ .
\]
Rephrasing this, using $v(n) = \frac{1}{n} \Var X_{0,n}$,
\[
v(k_n n) \geq a v(n) \text{ for } n \geq N'\ .
\]
We may further increase $N'$ so that $v(N')>0$.

To get a contradiction we iterate the above variance bound. Define a sequence of integers $(n_j)$ by $n_1 = N'$ and
\[
n_{j+1} = k_{n_j}n_j \text{ for } j \geq 1\ .
\]
Note that
\begin{equation}\label{eq: n_j_bound}
n_{j+1} \geq n_j^{1+2E+\epsilon/2} \text{ for all }j \geq 1\ , 
\end{equation}
and so $n_j \to \infty$ with $n_j \geq N'$ for all $j \geq 1$. Therefore
\[
v(n_{j+1}) \geq a v(n_j) \text{ for all } j \geq 1\ .
\]
If $\Lambda < 1$ then $a$ can be chosen larger than $1$ and so by iteration,
\[
v(n_j) \geq v(N') > 0 \text{ for all } j \geq 1\ ,
\]
which contradicts $\Var X_{0,n} = O\left( \frac{n}{(\log n)^\beta} \right)$.

Otherwise, $\Lambda \geq 1$ and
\[
A := \frac{\log(1/a)}{\log(1+2E+\epsilon/2)} > 0\ .
\]
Estimate using \eqref{eq: n_j_bound}:
\[
v(n_{j+1}) \left( \frac{\log n_{j+1}}{\log n_j} \right)^A  \geq v(n_{j+1}) (1+2E+\epsilon/2)^A \geq v(n_j)\ ,
\]
or $v(n_{j+1}) (\log n_{j+1})^A \geq v(n_j) (\log n_j)^A$. By iteration,
\begin{equation}\label{eq: pizza_pie}
\frac{\Var X_{0,n_{j+1}}}{\frac{n_{j+1}}{(\log n_{j+1})^A}} \geq \frac{\Var X_{0,N'}}{\frac{N'}{(\log N')^A}} >0 \text{ for all } j \geq 1\ .
\end{equation}
By definition, then $A \geq \beta$, and taking $a \uparrow \Lambda^{-1}$, we obtain a contradiction:
\[
\beta \leq \frac{\log \Lambda}{\log(1+2E +\epsilon/2)} < \beta\ .
\]

\section{First-passage percolation revisited }\label{sec:fpp}

In this section we verify that FPP satisfies association condition II and give a proof of Theorem~\ref{thm: second_fpp_thm}. The main ingredient is a large deviation bound on the radius of time-minimizing paths under no moment condition.  Theorem~\ref{thm:main1} will give us in this context:
\begin{thm}\label{thm: main_fpp_thm}
Assume \eqref{eq: percolation_assumption} and that for some $\delta>0$,
\begin{equation}\label{eq: fpp_chi}
\chi:= \liminf_n \frac{\log \|T(0,nx)-\mathbb{E}T(0,nx)\|_2}{\log n} = \limsup_n \frac{\log \|T(0,nx)-\mathbb{E}T(0,nx)\|_{2+\delta}}{\log n}\ .
\end{equation}
If $\chi < 1/2$ then
\[
\liminf_n \frac{\log (\mathbb{E}T(0,nx)-ng(x))}{\log n} \geq \chi\,  
\]
or, equivalently, for every $\epsilon >0$, there exists $C>0$ such that for all $n$
\[ Cn^{\chi-\epsilon} \leq \E T(0,nx)-ng(x). \] 
If $\Var T(0,nx) = O \left( \frac{n}{(\log n)^\beta} \right)$ for some $\beta>0$, then
\[
\limsup_n \frac{\log (\mathbb{E}T(0,nx)-ng(x))}{\log n} \geq \chi - \frac{1}{2} (2^{\beta^{-1}}-1)\ .
\]
\end{thm}

\subsection{Geodesic radius bound}

A path $\pi$ is a geodesic from $x$ to $y$ if $T(\pi) = T(x,y)$. Under \eqref{eq: percolation_assumption}, for all $x,y \in \mathbb{Z}^d$,
\[
\mathbb{P}(\exists \text{ a geodesic from } x \text{ to } y) = 1\ .
\]
(See \cite[(9.23)]{aspects}.) Write $G(0,x)$ for the union of all vertices in geodesics from $0$ to $x$. For a set $X \subset \mathbb{R}^d$, define $\text{diam }X = \sup_{x,y \in X} \|x-y\|_\infty$. The following theorem may be of independent interest.
\begin{thm}\label{thm: geodesic_length}
Assuming \eqref{eq: percolation_assumption}, there exist $M, \CC_{10} >0$ such that
\[
\mathbb{P}(\mathrm{diam }~G(0,x) \geq M\|x\|_\infty) \leq e^{-\CC_{10} \|x\|_\infty} \text{ for all } x \in \mathbb{Z}^d\ .
\]
\end{thm}

Under finite exponential moments, the proof of the above theorem is straightforward. With no moment assumption, we will need some percolation constructions. For $p \in [0,1]$ let $\mathbb{P}_p$ be the product measure on $\Omega = \{0,1\}^{\mathcal{E}^d}$ with marginal $\mathbb{P}(\omega(e) = 1) = p$, where $\omega$ is a typical element of $\Omega$. In a configuration $\omega$ we write $x \to y$ if there is a path from $x$ to $y$ with edges $e$ satisfying $\omega(e)=1$. This gives a connectivity equivalence relation and the equivalence classes are called open clusters. It is known that for $p > p_c$ there is almost surely a unique infinite open cluster. Define $B(n) = \{x \in \mathbb{Z}^d : \|x\|_\infty \leq n\}$ and $\partial B(n) = \{x \in \mathbb{Z}^d : \|x\|_\infty = n\}$.

\begin{lem}\label{lem: perc}
Let $A_n$ be the event that every path from 0 to $\partial B(n)$ intersects the infinite open cluster. There exists $p_0 \in (p_c,1)$ such that if $p \in [p_0,1]$ then for some $\CC_{11}>0$,
\[
\mathbb{P}_p(A_n) \geq 1-e^{-\CC_{11}n} \text{ for all } n \geq 1\ .
\]
\end{lem}
\begin{proof}
This result was essentially proved by Kesten \cite[Lemma~2.24]{aspects} and the next two paragraphs are mainly copied from there. Assign (random) colors to the vertices of $\mathbb{Z}^d$: $x$ is \emph{white} if all edges $e$ incident to $x$ have $\omega(e)=1$ and $x$ is \emph{black} otherwise. Next we need an auxiliary graph $\mathcal{L}$. The vertex set of $\mathcal{L}$ is the same as that of $\mathbb{Z}^d$. Two distinct vertices $u$ and $v$ are \emph{adjacent} on $\mathcal{L}$ -- and hence have an edge of $\mathcal{L}$ between them -- if $\|u-v\|_\infty = 1$. A set $S$ of vertices (edges) of $\mathcal{L}$ is \emph{connected} if for each $u_1,u_2 \in S$ ($e_1,e_2 \in S$), there exists a path on $\mathcal{L}$ whose first and last vertex (edge) are $u_1,u_2$ respectively ($e_1,e_2$ respectively). If $A$ is an $\mathcal{L}$-connected set of vertices we define $C(A,b)$, the \emph{black cluster} of $A$ on $\mathcal{L}$, as the union of $A$ and the set of all vertices $v_0$ of $\mathcal{L}$ for which there exists a path $(v_0,e_1, \ldots, e_n,v_n)$ on $\mathcal{L}$ from $v_0$ to some $v_n \in A$ such that $v_0, \ldots, v_{n-1}$ are all outside $A$ and black. Similarly, for a $\mathbb{Z}^d$-connected set of vertices $A$, $C(A,w)$, the \emph{white cluster} of $A$ on $\mathbb{Z}^d$, is the union of $A$ and the set of all vertices $v_0$ of $\mathbb{Z}^d$ for which there exists a path $(v_0, e_1, \ldots, e_n,v_n)$ on $\mathbb{Z}^d$ from $v_0$ to some $v_n \in A$ such that $v_0, \ldots, v_{n-1}$ are all outside $A$ and white. We shall write $C(v,w)$ for $C(\{v\},w)$. Note that we defined black clusters only on $\mathcal{L}$ and white clusters only on $\mathbb{Z}^d$. The color of vertices in $A$ has no influence on $C(A,b)$. By definition always $A \subset C(A,b)$. Similarly $A \subset C(A,w)$.

For an $\mathcal{L}$-connected set of vertices $C$ we define its exterior boundary as $\partial_{ext}C$, the set of vertices $v$ of $\mathcal{L}$ such that $v \notin C$ but $v$ is adjacent on $\mathcal{L}$ to some vertex $u$ in $C$ and there exists a path on $\mathbb{Z}^d$ from $v$ to $\infty$ which is disjoint from $C$. We next define shells $S(v)$. For any vertex $v = (v(1), \ldots, v(d))$ of $\mathbb{Z}^d$ (or $\mathcal{L}$), let 
\[
D_k(v) = [v(1)-k,v(1)+k] \times \cdots \times [v(d)-k,v(d)+k]
\]
and $n=n(v)$ as the minimal $k$ for which there exists a vertex $u \in D_k(v)$ with an infinite white cluster $C(u,w)$ on $\mathbb{Z}^d$. Set
\[
S(v) = \partial_{ext}C(D_{n(v)}(v),b)\ .
\]
The following properties of $n(v)$ and $S(v)$ hold.
\begin{enumerate}
\item \cite[Eq.~(2.22)]{aspects} : All vertices in $S(v)$ are white.
\item \cite[Lemma~2.23]{aspects} : If $C(D_{n(v)}(v),b)$ is finite, then $S(v)$ is $\mathbb{Z}^d$-connected. Moreover in this case $S(v)$ separates $v$ from $\infty$, in the sense that any path on $\mathbb{Z}^d$ from $v$ to $\infty$ must intersect $S(v)$.
\item \cite[Lemma~2.24]{aspects} : There exists $p_0 \in (p_c,1)$ (listed as $1-\pi_0$ in \cite[(2.20)]{aspects}) such that if $p > p_0$, then $C(D_{n(v)}(v),b)$ is almost surely finite and for some $K_2$,
\begin{equation}\label{eq: exponential_estimate}
\mathbb{P}(\text{diam }S(v) > k) \leq K_2 (k+3)^{d-1}3^{-k/4} \text{ for all } k \geq 0\ .
\end{equation}
\end{enumerate}
Because there is a vertex $u \in D_{n(0)}(0)$ with an infinite white cluster on $\mathbb{Z}^d$, this vertex must be contained in the infinite open cluster (in $(\omega(e))$); therefore, there is a vertex in $D_{n(0)}(0)$ that is connected to $\infty$ by an open path in $(\omega(e))$. Furthermore, no vertex of $D_{n(0)}(0)$ can be in $S(0)$, so we can connect $0$ to $u$ on $\mathbb{Z}^d$ without using a vertex of $S(0)$. This means that any path on $\mathbb{Z}^d$ from $u$ to $\infty$ must intersect $S(0)$ and so a vertex of $S(0)$ is connected to $\infty$ by an open path in $(\omega(e))$. Last, because $S(0)$ is $\mathbb{Z}^d$-connected, the definition of white implies that each vertex of $S(0)$ is connected to $\infty$ by an open path in $(\omega(e))$. So if $\text{diam }S(0) \leq k$ then each path on $\mathbb{Z}^d$ from $0$ to the set $\partial B(3k)$ must intersect $S(0)$ and thus the infinite open cluster. The estimate \eqref{eq: exponential_estimate} then completes the proof.
\end{proof}

We now recall the result of \cite[Theorem~1.1]{AP} and afterward prove Theorem \ref{thm: geodesic_length}.
\begin{lem}[Antal-Pisztora]\label{lem: AP}
Let $p>p_c$. Then there exists a constant $\rho = \rho(p,d) \in [1,\infty)$ such that
\[
\limsup_{\|y\|_\infty \to \infty} \frac{1}{\|y\|_\infty} \log \mathbb{P}_p ( d_I(0,y) > \rho\|y\|_\infty,~0,y \in I) < 0\ ,
\]
where $I$ is the infinite open cluster and $d_I$ is the intrinsic distance in $I$.
\end{lem}

\begin{proof}[Proof of Theorem~\ref{thm: geodesic_length}]
We choose $K>0$ such that $p:= \mathbb{P}(t_e \leq K) > p_0$ and define a percolation configuration $(\eta_e)$ from the weights $(t_e)$ by
\[
\eta_e = \begin{cases}
1 & \text{ if } t_e \leq K \\
0 & \text{ if } t_e > K
\end{cases}\ .
\]

\begin{figure}
\centering
 \includegraphics[scale=0.7]{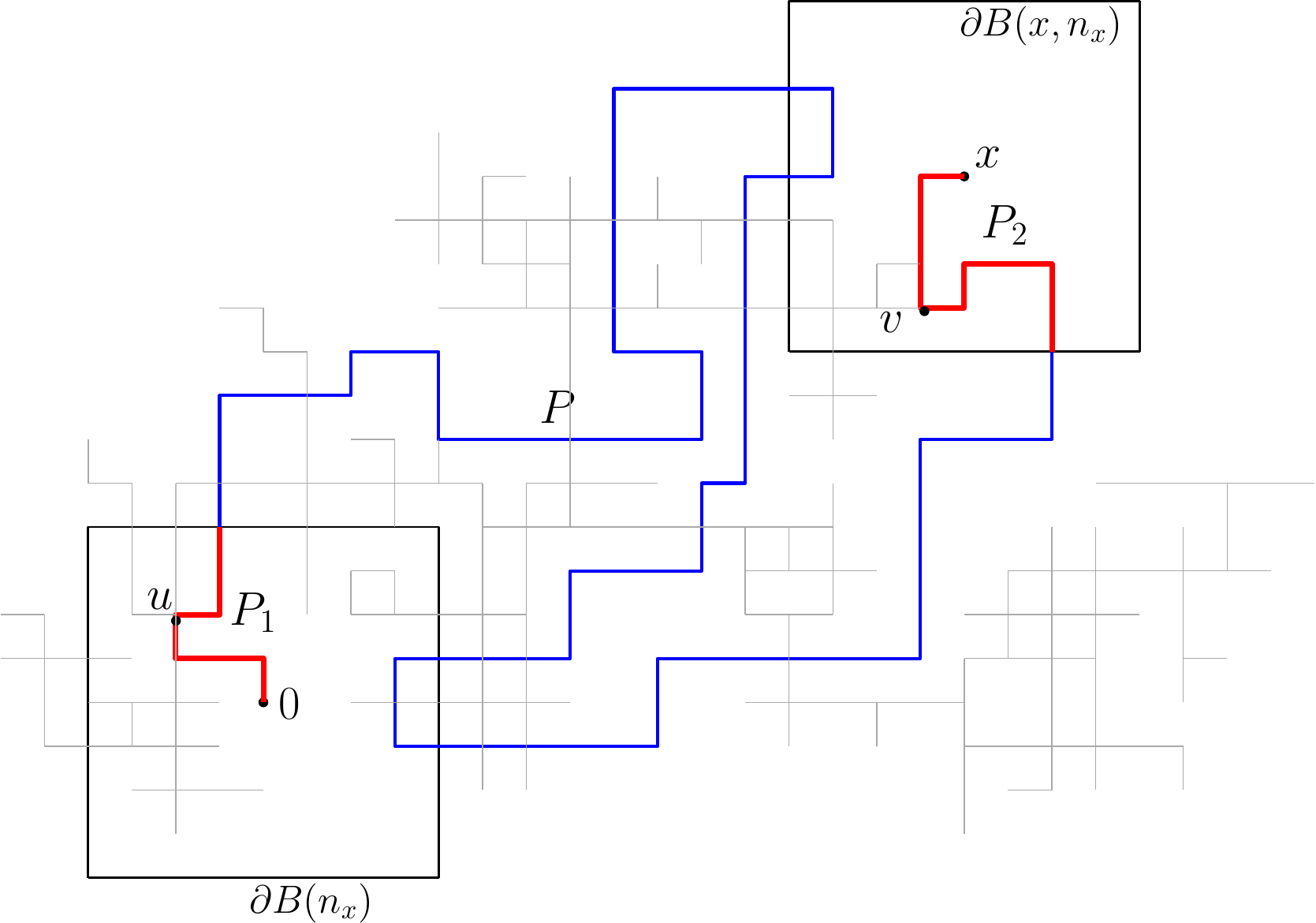}
 \caption{Construction of the geodesic $P$ from $0$ to $x$. The points $u$ and $v$ belong to the infinite open cluster (represented by the gray edges). The geodesic from $u$ to $v$ has diameter at  least $(M-1/2)\| x \|_{\infty}$.}
 \label{fig01a}
\end{figure}

For a given $x \in \mathbb{Z}^d$ with $\| x\|_{\infty} \geq 4$ and integer $M>0$, we first write $n_x = \lfloor \|x\|_\infty /4 \rfloor$ and estimate
\begin{equation}\label{eq: pepper}
\mathbb{P}(\text{diam }G(0,x) \geq 2M\|x\|_\infty) \leq 2e^{-\CC_{11} n_x} + \mathbb{P}(\text{diam } G(0,x) \geq  2M\|x\|_\infty,~A_{n_x},~B_{n_x})\ ,
\end{equation}
where $A_{n_x}$ is written for the event from Lemma~\ref{lem: perc} for the percolation configuration $(\eta_e)$ and $B_{n_x}$ is the same event with $0$ translated to $x$. On the event on the right we can select a self-avoiding geodesic $P$ from $0$ to $x$ which has diameter at least $M\|x\|_\infty$. Write $P_1$ for the portion of $P$ from 0 to its first intersection of $\partial B(n_x)$ and $P_2$ for the portion from its last intersection of $\partial B(x,n_x)$ to $x$ (here $B(x,n)$ is the translate of $B(n)$ centered at $x$). See Figure \ref{fig01a}. We can then choose $u,v$ vertices of $P_1$ and $P_2$ respectively such that $\|u\|_\infty \leq n_x$ and $\|v-x\|_\infty \leq n_x$ and both $u,v$ are in the infinite open cluster of $(\eta_e)$. By construction, the portion $P_3$ of $P$ from $u$ to $v$ is a geodesic that has diameter at least $(M-1/2)\|x\|_\infty$. We now apply Lemma~\ref{lem: AP} to find $\CC_{12}, \CC_{13}$ such that
\[
\mathbb{P}_p(d_I(0,y) > \rho\|y\|_\infty,~0,y \in I) \leq \CC_{13} e^{-\CC_{12}\|y\|_\infty} \text{ for all } y \in \mathbb{Z}^d\ .
\]
So for $u \in B(n_x)$ and $v \in B(x,n_x)$, as $2\|x\|_\infty \geq \|u-v\|_\infty \geq \|x\|_\infty/4$,
\[
\mathbb{P}_p(d_I(u,v) > 2\rho\|x\|_\infty,~u,v \in I) \leq \CC_{13} e^{-(\CC_{12}/4)\|x\|_\infty}
\]
and by a union bound,
\[
\mathbb{P}_p(d_I(u,v) > 2\rho\|x\|_\infty \text{ for some } u \in B(n_x) \cap I,~v \in B(x,n_x) \cap I) \leq \CC_{14}e^{-\CC_{15} \|x\|_\infty}\ 
\]
for some $\CC_{14}, \CC_{15}>0$.

On the complement of this event, each $u \in B(n_x) \cap I$ and $v \in B(x,n_x) \cap I$ have $d_I(u,v) \leq 2 \rho\|x\|_\infty$ and so $T(u,v) \leq 2K\rho \|x\|_\infty$. Use this in the right side of \eqref{eq: pepper} to bound it above by
\begin{equation}\label{eq: pepper_2}
\CC_{16} e^{-\CC_{15} \|x\|_\infty} + \mathbb{P}\left(\begin{array}{c}
\exists ~u \in B(n_x),v \in B(x,n_x) \text{ with } \text{diam }G(u,v) \text{ at}\\ \text{least } (M-1/2)\|x\|_\infty 
 \text{ but } T(u,v) \leq 2K\rho\|x\|_\infty
 \end{array}\right)\ .
\end{equation}

Last, we appeal to Kesten's result \cite[Proposition~5.8]{aspects}, which states that under \eqref{eq: percolation_assumption}, there exist constants $a,\CC_{17}>0$ such that for all $n \geq 1$,
\[
\mathbb{P}\bigg(\exists \text{ self-avoiding } \gamma \text{ starting at } 0 \text{ with } \#\gamma \geq n \text{ but with } T(\gamma) < an \bigg) \leq e^{-\CC_{17}n}\ .
\]
By a union bound, for all $n \geq 1$ and $x \in \mathbb{Z}^d$ with $\| x\|_{1} \geq 4$,
\[
\mathbb{P}\bigg( \exists \text{ self-avoiding } \gamma \text{ from }B(n_x) \text{ to } B(x,n_x) \text{ with } \#\gamma \geq n \text{ but } T(\gamma) < an \bigg) \leq (2n_x+1)^d e^{-\CC_{17} n}\ .
\]
If there exist $u \in B(n_x)$ and $v \in B(x,n_x)$ with diam $G(u,v) \geq (M-1/2)\|x\|_\infty$ then we may select a geodesic between $u$ and $v$ that has at least $(M/2-1/4)\|x\|_\infty$ edges. So fixing any $M$ with $4K\rho/a +1/2 < M$, the expression in \eqref{eq: pepper_2} is bounded by 
\[
\CC_{16} e^{-\CC_{15} \|x\|_\infty} + (2n_x+1)^d e^{-\CC_{17} (M/2-1/4)\|x\|_\infty}
\]
and this is bounded by $\CC_{18}e^{-\CC_{10}\|x\|_\infty}$ for some $\CC_{18}>0$, ending the proof of the theorem.

\end{proof}

\subsection{Verifying hypotheses for FPP}\label{sec:verhypFPP}
We begin by establishing suitable mixing for the sequence $\{X_{m,n}\}$ defined by $X_{m,n} = T(mx,nx)$.
\begin{prop}
Assuming \eqref{eq: percolation_assumption}, there exist $m_0,\CC_{19},\CC_{20}>0$ such that mixing coefficient $\alpha_n(m)$ from \eqref{eq:alphadef} satisfies
\[
\alpha_n(m) \leq \CC_{19}e^{-\CC_{20}mn \|x\|_\infty} \text{ for } n \geq 1,~m > m_0,~ x \in \mathbb{Z}^d\ .
\]
\end{prop}

\begin{proof}
We may assume $x \neq 0$. Letting $m,n \geq 1$ and $k \geq 0$, we seek a bound on $|\mathbb{P}(A\cap B)-\mathbb{P}(A)\mathbb{P}(B)|$, where 
\begin{equation}\label{eq: A_B_choice}
A \in \sigma(X_{in, (i+1)n}:\, 0 \leq i < k) \text{ and } B \in \sigma(X_{in,(i+1)n}:\, k+m \leq i )\ .
\end{equation}
To do this, we will replace the $X_{m,n}$'s with passage times using edges only on one side of a hyperplane. Let $H_k:= \{y \in \mathbb{R}^d: y \cdot x < (k+m/2)n x\cdot x \}$ and, in the case $i \leq k$, define
\[
Y_{i}:= \inf_{\pi \in H_k} T(\pi)\ ,
\]
where the infimum is over lattice paths connecting $in x$ to $(i+1) n x$ with vertices in $H_k$. Similarly, 
when $i \geq k+m$, let
\[
Y_{i}:= \inf_{\pi \in \mathbb{Z}^d \setminus H_k} T(\pi)\ .
\]
%where the infimum is over lattice paths connecting $in x$ to $(i+1) n x$ with vertices in $\mathbb{Z}^d \setminus H_k$.
We will bound the probability that $X_{in,(i+1)n} \neq Y_i$ for some $i$.

Define the event
\[
\Xi(k,m):= \left\{\text{there is some } i \in [0,k) \cup [k+m,\infty)\text{ such that } X_{in,(i+1)n} \neq Y_i\right\}
\]
and fix $\eta \in (0,1/2)$ (to be chosen later). The $\ell^2$-distance from $inx$ to the boundary of $H_k$ is $n|k+(m/2)-i|  \|x\|_2$, so on $\Xi(k,m)$, there is some $i$ such that a geodesic $\pi$ from $i n x$ to $(i+1) n x$ exits the ball $B_i^*$, where
\[
B_i^* = \left\{y \in \mathbb{R}^d : \|y-inx\|_2 \leq \eta n |k+m/2 -i| \|x\|_2 \right\}\ .
\]
Letting $v$ be the first vertex on $\pi$ in $(B_i^*)^c$, $\pi$ visits the boundary of $H_k$ after it visits $v$ (and therefore before $(i+1)nx$), so the segment of $\pi$ from $v$ to $(i+1) n x$ has diameter at least
\begin{equation}\label{eq:etaone}
(1-\eta)n\left|k+m/2 - i\right| \frac{\|x\|_2}{\sqrt d}\ .
\end{equation}
 \begin{figure}
\centering
 \includegraphics[scale=0.7]{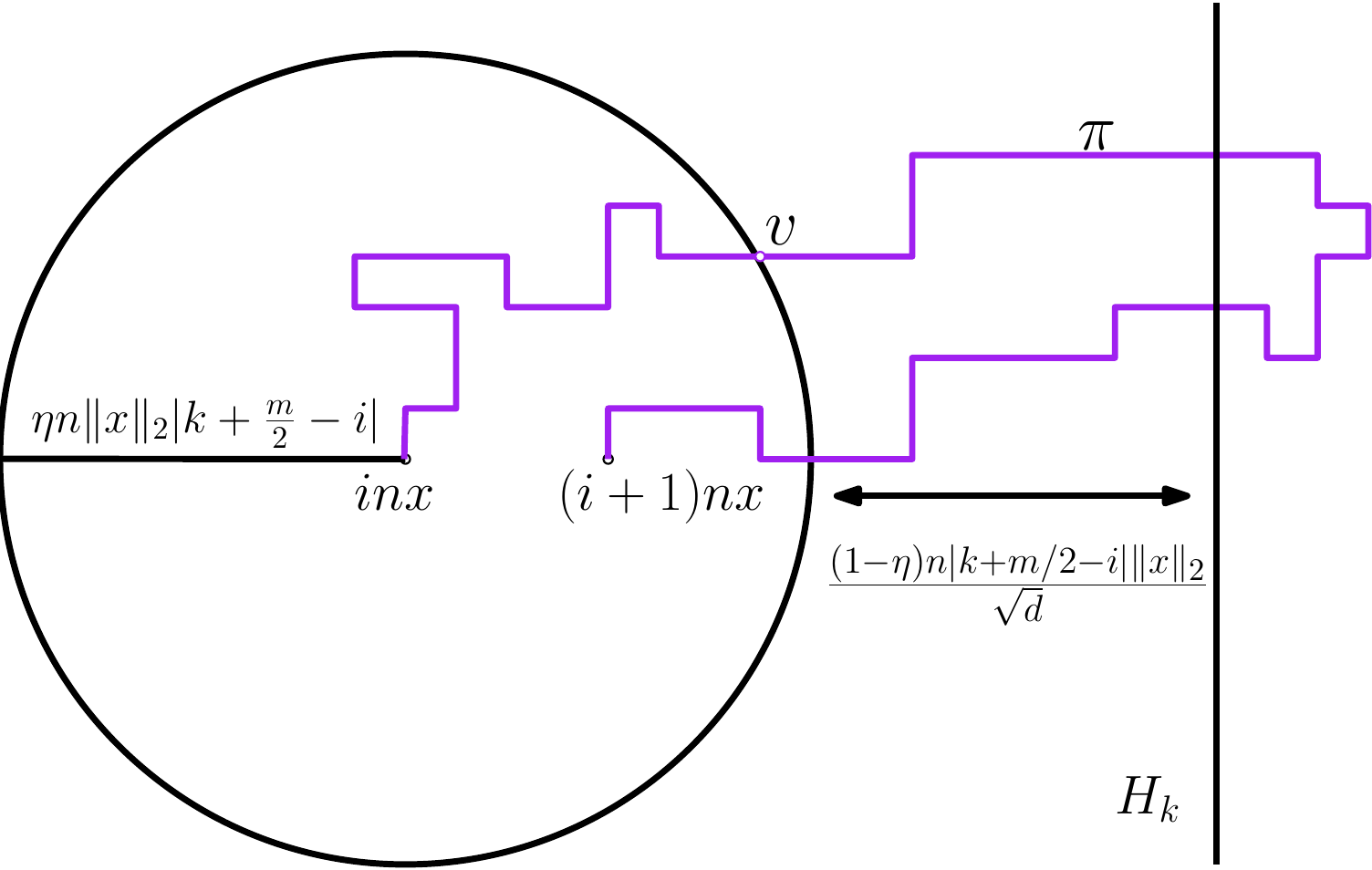}
 \caption{The event $\Xi(k,m)$: there exists $i$ such that the geodesic (in purple) from $inx$ to $(i+1)nx$ crosses the boundary of the half-space $H_k$. The segment of this geodesic from $v$ to $(i+1)nx$ has diameter at least $(1-\eta)n\left|k+m/2 - i\right| \frac{\|x\|_2}{\sqrt d}$.} 
 \label{fig02}
\end{figure}
(See Figure \ref{fig02} for an illustration.) Moreover,
\begin{equation}\label{eq:etatwo}
n( \eta  \left|k+m/2 - i \right|  - 1) \frac{\|x\|_2}{\sqrt d} \leq \|v-(i+1) n x\|_\infty \leq n ( \eta \left|k+m/2 - i \right|  + 1) \|x\|_2\ .
\end{equation}
Comparing \eqref{eq:etaone} and \eqref{eq:etatwo}, and using $|k+m/2-i| \geq m/2$ for $i \in [0,k] \cup [k+m,\infty)$, we see that if $\eta$ is sufficiently small and $m$ is larger than some $m_0$, then uniformly in $v,k, i, n$ and $x$,
\[
\text{diam }G(v,(i+1)nx) \geq M \|v - (i+1)n x\|_\infty \geq  \frac{M \eta}{2\sqrt{d}} \left| k+m/2 - i \right| n \|x\|_2\ , 
\]
where $M$ is from Theorem \ref{thm: geodesic_length}.
By a union bound, the probability that such a $v$ can be found for a fixed $i$ is bounded above by
\[
\CC_{21} \left(n\left|k+m/2-i\right| \|x\|_2 \right)^d \exp(-\CC_{22} n |k + m/2 -i| \|x\|_2 ) 
\]
for some $\CC_{21}, \CC_{22}>0$.
Thus, for all $k\geq 0,~m > m_0$ we obtain
\begin{align*}
\mathbb{P}(\Xi(k,m)) &\leq 2\CC_{21} \sum_{i=1}^\infty \left(n(m/2+i)\|x\|_2\right)^d \exp(-\CC_{22}n(m/2+i)\|x\|_2) \\
&\leq \CC_{19} \exp(-\CC_{20} mn\|x\|_2)\ .
\end{align*}

We use the above to produce a bound on $\alpha_n$. Let $k$ and $n$, as well as $m > m_0$, be arbitrary; let $A$ and $B$ be events as in \eqref{eq: A_B_choice}. Find Borel sets $E_A \subset \mathbb{R}^k$ and $E_B \subset \mathbb{R}^{\mathbb{N}}$ with 
\[
A = \{(X_{0,n}, \ldots, X_{(k-1)n,kn}) \in E_A\} \text{ and } B = \{(X_{(k+m)n,(k+m+1)n}, \ldots ) \in E_B\}\ .
\]
Define the corresponding events
\[
A' = \{(Y_0, \ldots, Y_{k-1}) \in E_A\} \text{ and } B' = \{(Y_{k+m}, \ldots) \in E_B\}
\]
and note that since they depend on the states of disjoint sets of edges, 
\[
|\mathbb{P}(A'\cap B') - \mathbb{P}(A')\mathbb{P}(B')| = 0\ .
\]
On $\Xi(k,m)$, the events $A$ and $A'$ are equal (and similarly for $B$ and $B'$). Therefore
\[
|\mathbb{P}(A\cap B) - \mathbb{P}(A' \cap B')| \leq \mathbb{P}(\Xi(k,m)) \leq \CC_{19} \exp(-\CC_{20}mn\|x\|_2)
\]
and
\begin{align*}
|\mathbb{P}(A)\mathbb{P}(B) - \mathbb{P}(A')\mathbb{P}(B')| & \leq \mathbb{P}(A) |\mathbb{P}(B)-\mathbb{P}(B')| + \mathbb{P}(B')|\mathbb{P}(A)-\mathbb{P}(A')| \\
&\leq |\mathbb{P}(B)-\mathbb{P}(B')| + |\mathbb{P}(A)-\mathbb{P}(A')| \\
&\leq 2\CC_{19}\exp(-\CC_{20}mn\|x\|_2)\ .
\end{align*}
Combining these two inequalities completes the proof.
\end{proof}

The above mixing result will be sufficient for \eqref{eq:covcond}. We will use a version of it proved below to verify association condition II. For this, we must make a moment assumption on $t_e$. So assume that
\begin{equation}\label{eq: fpp_moment}
\mathbb{E}t_e^{\frac{1}{d} + \eta}<\infty \text{ for some }\eta>0\ .
\end{equation}
This is sufficient to establish $\mathbb{E}Z^{2+d\eta}<\infty$, where $Z$ is the minimum of $2d$ i.i.d. variables distributed as $t_e$. By \cite[Lemma~3.1]{CD}, $\mathbb{E}T(x,y)^{2+d\eta}<\infty$ for all $x,y \in \mathbb{Z}^d$.

\begin{prop}\label{prop: cov_fpp} 
Assuming \eqref{eq: percolation_assumption} and \eqref{eq: fpp_moment}, there are constants $\CC_{23}, \CC_{24}>0$ such that for all $n\geq 1$, $x \in \mathbb{Z}^d$, and $r \geq m_0+1$,
\[
\max_{i \geq 1} \sum_{j : |i-j| \geq r}^\infty \mathrm{Cov} (X_{in,(i+1)n} X_{jn,(j+1)n}) \leq \CC_{23} \exp(-\CC_{24}nr \|x\|_\infty)\ .
\]
\end{prop}

\begin{proof}
For the proof we require a lemma \cite[Eq.~(2.2)]{Dav}.

\begin{lem}[Davydov]
\label{lem:orthogcontrol}
Let $f$ be a $\sigma(X_{in, (i+1)n}:\, 0 \leq i < k)$-measurable random variable and let $g$ be a $\sigma(X_{in,(i+1)n}:\, k+m \leq i )$-measurable random variable. Assume that $\mathbb{E}|f|^p$ and $\E |g|^q$ are finite for nonnegative $p,q$ such that $1/p + 1/q < 1$. Then \[
|\E fg - \E f \E g| \leq 12 \|f\|_p \|g\|_{q} \alpha_n(m)^{1 - 1/p - 1/q}\ .
\]
\end{lem}

We apply this lemma with $p=q=2+\delta$, where $\delta = d\eta$. This makes $1-1/p-1/q = \frac{\delta}{2+\delta}$, so if $|i-j| = m \geq m_0+1$,
\begin{align*}
\text{Cov} (X_{in,(i+1)n}, X_{jn,(j+1)n}) &\leq 12 \|X_{0,n}\|_{2+\delta}^2 \alpha_n(m)^{\frac{\delta}{2+\delta}} \\
&\leq \CC_{25} \|T(0,nx)\|_{2+\delta}^2 \exp(-\CC_{26} mn\|x\|_\infty)\ 
\end{align*}
for some $\CC_{25}, \CC_{26}>0$.
Let $\pi = x_0, \ldots, x_{n\|x\|_1}$ be a deterministic lattice path of $n\|x\|_1$ number of edges from $0$ to $nx$ and use translation invariance for
\[
\|T(0,nx)\|_{2+\delta} \leq \sum_{k=0}^{n\|x\|_1-1} \|T(x_k,x_{k+1})\|_{2+\delta} = n\|x\|_1 \|T(0,e_1)\|_{2+\delta}\ .
\]
So the above covariance is bounded by
\[
\CC_{27} n^2\|x\|_\infty^2 \exp(-\CC_{26}mn\|x\|_\infty) \leq \CC_{28}\exp(-\CC_{29}mn\|x\|_\infty)\ .
\]
Summing over $m \geq r$ completes the proof.
\end{proof}

Now we can prove Theorem~\ref{thm: main_fpp_thm}.
\begin{proof}[Proof of Theorem~\ref{thm: main_fpp_thm}]
Assume that for some $\delta>0$, $\chi:= \underline{\chi}_2 = \overline{\chi}_{2+\delta}\leq 1/2$. This implies that $\|T(0,nx)\|_{2+\delta}<\infty$ for some $n$ and, since $T(0,nx) \geq \min\{t_{e_1}, \ldots, t_{e_{2d}}\}$, the minimum of the $2d$ edges with endpoint 0, we see that assumption \eqref{eq: fpp_moment} must hold for some $\eta>0$. Therefore $\|T(0,nx)\|_{2+\delta}<\infty$ for all $n$. 

We first show that if the distribution of $t_e$ is non-degenerate, then $\Var T(0,x) \neq 0$ for all $x\neq 0$. Let 
\begin{align*}
A_l &= \{t_e \in A \text{ for all } e \text{ with endpoints in } B(l)\} \\
B_l &= \{t_e \in B \text{ for all } e \text{ with endpoints in } B(l)\}\ ,
\end{align*}
where $A$ and $B$ are two bounded subsets of $\mathbb{R}$ such that $a:= \sup A < \inf B =: b$ and $\mathbb{P}(t_e \in A)>0$, $\mathbb{P}(t_e \in B)>0$. Then for $l= \|x\|_\infty$,
\[
\mathbb{P}(T(0,x) \geq b\|x\|_1) \geq \mathbb{P}(B_l) > 0 \text{ and } \mathbb{P}(T(0,x) \leq a\|x\|_1) \geq \mathbb{P}(A_l)>0\ ,
\]
so the distribution of $T(0,x)$ is non-degenerate.

Since $\{X_{m,n}\}$ satisfies assumptions 1-3, we need only verify association condition II to apply Theorem~\ref{thm:main1}. By the Harris-FKG inequality \cite[Theorem~2.15]{BLM} any coordinatewise increasing functions $f_1,f_2:[0,\infty)^{\mathcal{E}^d} \to \mathbb{R}$ satisfy $\text{Cov}(f_1,f_2)\geq 0$ whenever this covariance exists. Since $T(x,y)$ is an increasing function of the edge weights for each $x,y$, this shows that the variables $\{X_{0,n}, X_{n,2n}, \ldots\}$ are positively associated for each $n$.

To show \eqref{eq:covcond}, note that we can assume $\chi > -\infty$, so that $\Var T(0,nx) \geq n^{-\CC_{30}}$ for all $n$ and some constant $\CC_{30}$. Therefore if $r \geq m_0+1$, use Proposition~\ref{prop: cov_fpp} for
\begin{align*}
\frac{\max_{i \geq 1} \sum_{j : |i-j| \geq r} \mathrm{Cov} (X_{in,(i+1)n} X_{jn,(j+1)n})}{\Var X_{0,n}} &\leq \CC_{23}n^{\CC_{30}} \exp\left( - \CC_{24}nr\|x\|_\infty \right) \\
&\leq \CC_{31}\exp \left( -\CC_{32}nr\|x\|_\infty \right)\ .
\end{align*}
So we define
\begin{equation}\label{eq: u_def}
u(r) = \begin{cases}
\CC_{31} \exp \left( -\CC_{32} r\|x\|_\infty \right) & \text{ if } r \geq m_0+1 \\
2m_0+1 + \CC_{31} & \text{ otherwise}
\end{cases}\ ,
\end{equation}
and \eqref{eq:covcond} holds.
\end{proof}

\subsubsection{Proof of Theorem~\ref{thm: second_fpp_thm}}\label{sec: second_fpp}

%{\tt Check if we can do this with just $p$ moments.} 
The proof is similar to that of Theorem~\ref{thm:main1} in the case $\Var X_{0,n} = O(n/(\log n)^\beta)$. We will set $\beta=1$, due to \cite[Theorem~1.1]{DHS14}:
\begin{lem}[Damron-Hanson-Sosoe]
Let $d\geq 2$. Under \eqref{eq: percolation_assumption} and \eqref{eq: exp_assumption}, there exist $c_1,c_2>0$ such that for all $x \in \mathbb{Z}^d$ with $\|x\|_1 > 1$,
\[
\mathbb{P}\left( |T(0,x) - \mathbb{E}T(0,x)| \geq \frac{\|x\|_1^{1/2}}{(\log \|x\|_1)^{1/2}} \lambda \right) \leq c_1 e^{-c_2 \lambda} \text{ for } \lambda \geq 0\ .
\]
\end{lem}
Given a nonzero $x \in \mathbb{Z}^d$, we define as before $X_{m,n} = T(mx,nx)$ and note that by the above lemma, one has
\begin{equation}\label{eq: variance_fpp}
\Var X_{0,n} = O\left( \frac{n}{\log n} \right) \text{ and } \overline{\chi}_p \leq 1/2 \text{ for all } p > 0\ .
\end{equation}
%Furthermore, by the exponential moment assumption, \eqref{eq: fpp_moment} holds for all $\eta>0$ and therefore association condition II holds, allowing us to apply Corollary~\ref{corollary: main_cor}. 
We need one more result, which is \cite[Eq.~(1.13)]{Kesten}.
\begin{lem}[Kesten]\label{lem: kesten_lower_bound}
Assume \eqref{eq: percolation_assumption} holds and that the distribution of $t_e$ is not concentrated on one point. If $\mathbb{E}t_e^2<\infty$ then for some $C>0$, $\Var T(0,ne_1) \geq C$ for all $n$.
\end{lem}

\noindent
Although this lemma is stated for the passage time in the $e_1$ direction, its proof immediately gives
\begin{equation}\label{eq: var_lower_bound}
\Var T(0,nx) \geq C > 0 \text{ for all nonzero } x \in \mathbb{Z}^d \text{ and } n \geq 1\ .
\end{equation}

Last, by our assumption $\mathbb{E}e^{\alpha t_e}<\infty$, \eqref{eq: fpp_moment} holds for all $\eta>0$ and therefore we can apply Proposition~\ref{prop: cov_fpp}, along with \eqref{eq: var_lower_bound} to bound, as in the last proof,
\[
\frac{\max_{i \geq 1} \sum_{j : |i-j| \geq r} \mathrm{Cov} (X_{in,(i+1)n} X_{jn,(j+1)n})}{\Var X_{0,n}} \leq \CC_{33}\exp \left( -\CC_{34}nr\|x\|_\infty \right)\ .
\]
This shows association condition II with $u(r)$ defined as in \eqref{eq: u_def}.

Given these tools we proceed with the proof of Theorem~\ref{thm: second_fpp_thm}. So assume for a contradiction that $\overline{\gamma} < -1/2$. By \eqref{eq: var_lower_bound}, for any $\epsilon>0$ small enough, there exists $N \geq 2$ such that
\[
(\mathbb{E}T(0,nx) - ng(x))^2 < n^{-1 - \epsilon} \Var T(0,nx) \text{ for } n \geq N\ .
\]
Next pick $\delta$ so large that
\begin{equation}\label{eq: delta_condition}
1/2 < (1+\epsilon/2)\left(\frac{1}{2} - \frac{1}{2+\delta} \right)\ .
\end{equation}
Setting
\[
k_n = \min \left\{ \left\lceil \frac{\Var T(0,nx)}{(\mathbb{E}T(0,nx) - ng(x))^2} n^{-\epsilon/2} \right\rceil, \left\lceil n^{1+\epsilon/2} \right\rceil \right\}\ ,
\]
and $\Delta(\delta) = \frac{1}{2} - \frac{1}{2+\delta}$, one has
\begin{equation}\label{eq: k_n_lower_bound}
k_n^\Delta \geq n^{(1 + \epsilon/2)\Delta} \text{ for } n \geq N\ .
\end{equation}
So using \eqref{eq: variance_fpp}, \eqref{eq: var_lower_bound} and \eqref{eq: delta_condition},
\[
\frac{\|T(0,nx) - \mathbb{E}T(0,nx)\|_{2+\delta}}{\|T(0,nx) - \mathbb{E}T(0,nx)\|_2} = o(k_n^\Delta)\ .
\]
We may now apply Corollary~\ref{corollary: main_cor} to deduce that with $l_n=n$ and any positive $a<1/2$, there is an $N' \geq N$ such that
\[
\Var T(0,k_n n x) \geq a k_n \Var T(0,nx) \text{ for } n \geq N'
\]
or, setting $v(0,n) = \frac{1}{n} \Var T(0,nx)$,
\[
v(k_n n) \geq a v(n) \text{ for } n \geq N'\ .
\]

Now iterate this bound, defining a sequence $(n_j)$ by $n_1=N'$ and $n_{j+1} = k_{n_j}n_j$ for $j \geq 1$. Then
\begin{equation}\label{eq: n_j_condition_2}
n_{j+1} \geq n_j^{2+\epsilon/2} \text{ for all } j \geq 1\ ,
\end{equation}
so $n_j \to \infty$ with $n_j \geq N'$ for all $j \geq 1$. This gives $v(n_{j+1}) \geq a v(n_j)$ for all $j \geq 1$. Set
\[
A = \frac{\log (1/a)}{\log (2+\epsilon/2)} > 0
\]
and as in \eqref{eq: pizza_pie},
\[
\frac{\Var T(0,n_{j+1}x)}{\frac{n_{j+1}}{(\log n_{j+1})^A}} \geq \frac{\Var T(0,N'x)}{\frac{N'}{(\log N')^A}} > 0 \text{ for all } j \geq 1\ .
\]
Due to \eqref{eq: variance_fpp}, $A \geq 1$, meaning
\[
\log(1/a) \geq \log(2+\epsilon/2)\ .
\]
Take $a \uparrow 1/2$ to get a contradiction.

\qed

\appendix
\section{CLT under mixing conditions}\label{sec: mixing_clt}

In this appendix, we present the proof of Lemma~\ref{lem: alpha_CLT}. Recall that if $\Sigma_1$ and $\Sigma_2$ are sigma-algebras, we write $\alpha(\Sigma_1,\Sigma_2) = \sup_{A \in \Sigma_1, B \in \Sigma_2} | \mathbb{P}(A\cap B)-\mathbb{P}(A)\mathbb{P}(B)|$.

\begin{thm}[See \cite{brad}, Theorem 1.14]
\label{thm:baseclt}
Suppose for each positive integer $n$, the following four statements hold (where $\mathbb{R}$ is taken to have the Borel sigma-algebra):
\begin{enumerate}
\item $k(n)$ is a positive integer.
\item $X_1^{(n)},\,\ldots,\,X_{k(n)}^{(n)}$ are real-valued random variables.
\item $Y_1^{(n)},\,\ldots,\,Y_{k(n)}^{(n)}$ are  independent real-valued random variables.
\item For each $k = 1, \,2,\, \ldots, \, k(n)$, the two random variables $X_k^{(n)}$ and $Y_k^{(n)}$ have the same distribution.
\end{enumerate}
Suppose also that
\[\lim_n \sum_{k=1}^{k(n)-1} \alpha \left(\sigma(X_j^{(n)},\,1 \leq j \leq k), \, \sigma(X_{k+1}^{(n)})\right) = 0. \]
Finally, suppose $\mu$ is a probability measure on $\mathbb{R}$. The following two statements are equivalent:
\begin{enumerate}
\item[(a)] $X_1^{(n)} + \ldots + X_{k(n)}^n \Rightarrow \mu$ as $n \rightarrow \infty$.
\item[(b)] $Y_1^{(n)} + \ldots + Y_{k(n)}^n \Rightarrow \mu$ as $n \rightarrow \infty$.
\end{enumerate}
\end{thm}

%{\tt check condition on epsilon in next thm}

\begin{thm}[\cite{sotmal} as formulated in \cite{kuelphil}]
\label{thm:momineq}
Let $\{\eta_n,\,n\geq 1\}$ be a sequence of mean-zero random variables with $2 + \delta$th moment (with $0 < \delta \leq 1$) uniformly bounded by $1$. 
Let
\[\alpha(n) = \sup_k \sup_{A,B} \left| \mathbb{P}(A\cap B) - \mathbb{P}(A) \mathbb{P}(B)\right|, \]
where the latter supremum is over $A \in \sigma(\eta_1, \ldots, \eta_k)$ and $B \in \sigma(\eta_{k+n}, \ldots)$.
Suppose that there exists some $C_\eta >0$ and  $\varepsilon$ such that, for all $n$, $\alpha(n) \leq C_\eta n^{-(1+\varepsilon)(1 + 2/\delta)}$. 
There exists an $\alpha \in (0,\delta]$ depending only on $\varepsilon,\,\delta$ such that, for all $a> 0$,
\[ \E \left| \sum_{j=a+1}^{a+n} \eta_j\right|^{2 + \alpha} \leq C'_\eta n^{1 + \alpha/2}.\]
Here the constant $C_\eta ' > 0$ depends only on $\varepsilon, \, \delta,$ and $C_\eta$.
\end{thm}

%\begin{thm}\label{thm: BE}
%Let $\{Y_{k,n} : n \geq 1,~ k = 1, \ldots, r_n\}$ be an array with $\{Y_{1,n}, \ldots, Y_{r_n,n}\}$ independent for each $n$ and mean zero. Define $S_n = Y_{1,n} + \cdots + Y_{r_n,n}$ and $s_n^2 = \Var S_n$. If for each $\epsilon>0$,
%\[
%\frac{1}{s_n^2} \sum_{k=1}^{r_n} \mathbb{E} Y_{k,n}^2 \mathbf{1}_{\{|Y_{k,n}| \geq \epsilon s_n\}} \to 0\ ,
%\]
%then $S_n/s_n \Rightarrow N(0,1)$, the distribution of a standard Gaussian.
%\end{thm}

\begin{lem}
\label{lem:covbd}
Let $\{Y_k^{(n)}\}_{k,n}$ be an array of mean zero random variables, and let $\delta > 0$ be such that:
\begin{itemize}
\item For each $n$, the sequence $(Y_k^{(n)})_k$ is stationary;
\item For each $n$, $\E \left| Y_1^{(n)}\right|^2 = 1$ and $\E \left| Y_1^{(n)} \right|^{2 +\delta} < \infty$;
\end{itemize}
Then for all $k$ and $n$,
\begin{equation}
\label{eq:growcovbd}
\left| \mathrm{Cov}(Y_1^{(n)},Y_k^{(n)}) \right| \leq 12 \|Y_1^{(n)}\|_{2 + \delta}^2\, \alpha_n(k-1)^{\delta / (2 + \delta)}. \end{equation}

\end{lem}
\begin{proof}
By Lemma~\ref{lem:orthogcontrol}, we have
\begin{align*}
\mathrm{Cov}(Y_1^{(n)},Y_k^{(n)}) = \mathbb{E}Y_1^{(n)} Y_k^{(n)} \leq 12 \|Y_1^{(n)}\|_{2 + \delta}^2\, \alpha_n(k-1)^{\delta / (2 + \delta)}.
\end{align*}
Using a similar argument on $- \E Y_1^{(n)} Y_k^{(n)}$ proves \eqref{eq:growcovbd}.
\end{proof}

%\begin{thm}
%Let $\{X_i^{(n)}, i = 1, 2, \ldots, \, n = 1,\, 2,\, \ldots\}$ be an array of random variables such that $(X_i^{(n)})_i$ is a strictly stationary sequence for each $n$.  Assume that there exist $C_1, C_2$ as well as constants $\delta>0, a \geq 0$ such that, for each $n$,
%\begin{enumerate}
%\item $\mathbb{E} X_1^{(n)} = 0,\, \E \left(X_1^{(n)}\right)^2 = 1$;
%\item $C_1 b \leq \mathbb{E} \left| \sum_{i=1}^b X_i^{(n)} \right|^2$ for all $b$;
%\item $\mathbb{E} \left| \sum_{i \in \mathfrak{B}} X_i^{(n)} \right|^2 \leq C_2 b$ for all $b$ and sets $\mathfrak{B}$ of indices such that $\# \mathfrak{B} = b$. {\tt this can be done with mixing if we impose alpha-mixing rate increase with $n$}
%\item $\mathbb{E} \left|X_1^{(n)} \right|^{2+\delta} <\infty$;
%\item For each $\kappa > 0$, there exists $C_3$ such that, for all $n$ and $x$, $\alpha_n(x) \leq C_3 x^{-\kappa}.$
%\end{enumerate}
%Then, letting $\sigma(n) = \Var \sum_{j=1}^{r(n)} X_j^{(n)},$ there exists $\Delta > 0$ such that
%\begin{equation}
%\label{eq:cltsum}
%\lim_{n \rightarrow \infty} \frac{1}{\sigma(n)}\sum_{j=1}^{r(n)} X_j^{(n)} \Rightarrow N(0,1)
%\end{equation}
%for any increasing $r(\cdot)$ such that  $\|X_1^{(n)}\|_{2+\delta} / r(n)^{\Delta} \rightarrow 0$ as $n \rightarrow \infty$.
%\end{thm}

\begin{proof}[Proof of Lemma~\ref{lem: alpha_CLT}]

We will argue via blocking, with large blocks of size $p(n)$ and small blocks of size $q(n)$. For each value of $r(n)$, let $p(n) = \lfloor r(n)^{1/2} \rfloor$  and choose $q(n) = \lfloor p(n)/\log r(n) \rfloor$. Define $k(n) = \lceil r(n) / (p(n) + q(n)) \rceil$.

We break the sum in \eqref{eq:cltsum} into blocks as follows. Define ``large blocks'' $\{V_i^{(n)}\}_{i=1}^{k(n)}$ and ``small blocks'' $\{W_i^{(n)}\}_{i=1}^{k(n)}$ inductively by letting
\begin{align*}
V_1^{(n)} = \sum_{j=1}^{p(n)} \eta_j^{(n)}; \quad W_1^{(n)} = \sum_{j=p(n)+1}^{p(n)+q(n)} \eta_j^{(n)}; \quad V_2^{(n)} = \sum_{j=p(n)+q(n)+1}^{2p(n) + q(n)} \eta_j^{(n)},
\end{align*}
and so on. Note that each $V_i^{(n)}$ is made up of $p(n)$ terms and each $W_i^{(n)}$ is made up of $q(n)$ terms. Now, notice
\[A(n):=\sum_{j=1}^{k(n)(p(n)+q(n))} \eta_j^{(n)} = \sum_{j=1}^{k(n)} V_j^{(n)} +  \sum_{j=1}^{k(n)} W_j^{(n)}.\]

\begin{clam}
$A(n)/\sqrt{\Var A(n)} \Rightarrow N(0,1)$ if and only if $\frac{1}{\sigma(n)}\sum_{j=1}^{r(n)} \eta_j^{(n)} \Rightarrow N(0,1)$.
\end{clam}
\begin{proof}
Note that $A(n) - \sum_{j=1}^{r(n)}\eta_j^{(n)}$ has the same distribution as $\sum_{j=1}^{b} \eta_j^{(n)}$ for some $b \leq p(n) + q(n)$. In particular,
\[\E \left[ A(n) -  \sum_{j=1}^{r(n)} \eta_j^{(n)} \right]^2 \leq 2 \CC_{4} p(n). \]
In particular, $\sqrt{\Var A(n)} / \sigma(n) \rightarrow 1$ and $\left(A(n) - \sum_{j=1}^{r(n)} \eta_j^{(n)}\right) / \sigma(n) \rightarrow 0$ in probability. Slutsky's Theorem completes the proof.
\end{proof}
The above claim allows us to assume that $r(n)$ has the form $k(n)(p(n)+q(n))$ for some integer $k(n)$, where $p(n) \in [r(n)^{1/2}/2, 2 r(n)^{1/2}]$ and $q(n) \in [r(n)^{1/2}/4 \log r(n), 4 r(n)^{1/2} / \log r(n)]$.
We will henceforth treat this as an assumption along with the other hypotheses of the lemma. Indeed, in replacing our original $r(n)$ by $k(n)(p(n) + q(n)),$ we preserve the assumption on the growth
rate of $r(n)$.

\begin{clam}\label{clam:probcon}
$\frac{1}{\sigma(n)} \sum_{j=1}^{k(n)} W_j^{(n)} \rightarrow 0$ in probability as $n \rightarrow \infty$.
\end{clam}
\begin{proof}[Proof of Claim \ref{clam:probcon}]
By the third assumption in the statement of the lemma, 
\[ \mathbb{E}\left[ \sum_{j=1}^{k(n)} W_j^{(n)} \right]^2 = o(r(n)). \]
However, by assumption, $\sigma(n)^2 \geq \CC_5 r(n)$. Chebyshev's inequality completes the proof.
\end{proof}

By Claim \ref{clam:probcon} combined with Slutsky's Theorem, it suffices to show that
\[\frac{1}{\sigma(n)} \sum_{j=1}^{k(n)} V_j^{(n)} \Rightarrow N(0,1). \]
To this end, let $\{U_j^{(n)},\, j=1,\ldots,\,k(n),\, n = 1,\,2,\ldots\}$ be an array of independent random variables such that $U_j^{(n)}$ and $V_j^{(n)}$ are identically distributed for all $j$ and $n$.

\begin{clam}
\label{clam:indrep}
\[\frac{1}{\sigma(n)} \sum_{j=1}^{k(n)} V_j^{(n)} \Rightarrow N(0,1) \text{ if and only if }  \frac{1}{\sigma(n)} \sum_{j=1}^{k(n)} U_j^{(n)} \Rightarrow N(0,1).\]
\end{clam}
\begin{proof}[Proof of Claim \ref{clam:indrep}]
We apply Theorem \ref{thm:baseclt}. The hypothesis on mixing becomes
\[k(n) \alpha_n(q(n))  \rightarrow 0 \text{ as } n \rightarrow \infty.\]
For each $\varepsilon > 0$, we have $k(n) < r(n)^{1/2+\varepsilon}$ and $q(n) > r(n)^{1/2-\varepsilon}$ for all $n$ sufficiently large. Since $\alpha_n(x)$ is bounded above by 
a constant multiple of $x^{-\kappa}$ uniformly in $n$, where $\kappa > 1,$ the hypothesis holds.
\end{proof}

It remains to show that the desired limiting behavior holds for the $U$-variables. To this end, defining $\sigma_U(n)^2:=\mathbb{E} \left(\sum_{j=1}^{k(n)} U_j^{(n)}\right)^2 = \mathbb{E} \sum_{j=1}^{k(n)} \left(U_j^{(n)}\right)^2$, then $\sigma(n) / \sigma_U(n) \rightarrow 1$ (we will delay proving this fact until the end of the proof). So we must show
\begin{equation}
\label{eq:uclt}
\frac{1}{\sigma_U(n)} \sum_{j=1}^{k(n)} U_j^{(n)} \Rightarrow N(0,1)\ ,
\end{equation}
and we will apply the Lyapunov condition: for some $q>1$ we would like
\[
\frac{1}{\sigma_U(n)^{2q}} \sum_{j=1}^{k(n)}\mathbb{E}|U_j^{(n)}|^{2q} \to 0\ .
\]
First apply Theorem~\ref{thm:momineq} to find $q>1$ with $2q\leq 2+\delta$ such that
\begin{equation}\label{eq: q_bound}
\|U_1^{(n)}\|_{2q} \leq C' p(n)^{1/2} \|\eta_1^{(n)}\|_{2+\delta}
\end{equation}
for some $C'$. For this $q$, we use stationarity and the assumption $\mathbb{E}\left(U_1^{(n)}\right)^2 \geq \CC_{5}p(n)$ to reduce the Lyapunov condition to
\[
\frac{1}{k(n)^{q-1}} \|\eta_1^{(n)}\|_{2+\delta}^{2q} \to 0\ .
\]
By definition of $k(n)$, this is satisfied if $\|\eta_1^{(n)}\|_{2+\delta} = o(r(n)^\Delta)$, where $\Delta$ is chosen as $\frac{1}{4}(1-q^{-1})$.

Last, to show that $\sigma(n)/\sigma_U(n) \rightarrow 1$, use Lemma~\ref{lem:orthogcontrol} and \eqref{eq: q_bound}:
\begin{align*}
\left|\mathbb{E} \sum_{j=1}^{k(n)} \left(U_j^{(n)}\right)^2 - \E \left[\sum_{j=1}^{k(n)} V_j^{(n)} \right] ^2 \right| &= \left| \sum_{\stackrel{i,j=1}{i \neq j}}^{k(n)} \text{Cov }(V_j^{(n)},V_i^{(n)}) \right| \\
&\leq 24 k(n)  \|U_1^{(n)}\|_{2q}^2 \sum_{m=1}^{\infty}  
\alpha_n(m q(n) )^{(q-1)/q}\\
&\leq 24 C'^2 r(n) \|\eta_1^{(n)}\|_{2+\delta}^2 \sum_{m=1}^{\infty} \alpha_n(C''m r(n)^{1/2-\varepsilon})^{(q-1)/q}
\end{align*}
for every $\varepsilon > 0$. Using the assumption on $\alpha_n$, this converges to 0 as $n \to \infty$ and shows $|\sigma(n)^2 - \sigma_U(n)^2| \to 0$. Last,
\[
\left| \frac{\sigma(n)}{\sigma_U(n)} - 1 \right| = \frac{|\sigma(n)^2-\sigma_U(n)^2|}{\sigma_U(n)(\sigma(n)+\sigma_U(n))} \to 0\ .
\]
\end{proof}

\section{CLT under association condition II}\label{appendix: CG}
Since the proof of Lemma~\ref{lem: CG} is almost the same as in \cite{CG}, we just give an outline, indicating the one main change to adapt the argument to our setting. The proof is another ``blocking'' argument and we begin with several definitions.

Note first that by Jensen's inequality, $\|\eta_i^{(n)}\|_{2+\delta}\geq 1$ for all $i,n$, so that $r(n) \to \infty$ as $n \to \infty$. Let $\ell$ be a positive integer and make the following definitions: $S(n) = \sum_{i=1}^{r(n)} \eta_i^{(n)}$,
\begin{align*}
m&= \lfloor r(n)/\ell\rfloor\ , \\
Y_j^{(n)} &= \sum_{(j-1)\ell < i \leq j\ell} \eta_i^{(n)} \text{ for } 1 \leq j \leq m\ , \\
S(n,\ell) &= \sum_{1 \leq j \leq m} Y_j^{(n)},~~Z(n) = S(n) - S(n,\ell)\ , \\
\sigma(n,\ell)^2 &= \Var S(n,\ell),~~ s(n,\ell)^2 = \sum_{1 \leq j \leq m} \Var Y_j^{(n)}\ , \\
\sigma(n)^2 &= \Var S(n)\ , \\
\phi_n(t) &= \mathbb{E}e^{itS(n)},~~ \phi_{n,\ell}(t) = \mathbb{E}e^{itS(n,\ell)}\ , \\
\phi_{n,j}(t) &= \mathbb{E}e^{itY_j^{(n)}}\ .
\end{align*}
Note that $Y_j^{(n)}$ and $Z(n)$ depend upon the choice of $\ell$, and that $m=m(n) \to \infty$ as $n \to \infty$.

We state a preliminary lemma.
\begin{lem}\label{lem: 2.1}
\[
\limsup_n \frac{\sigma(n)^2}{s(n,\ell)^2} \leq 1 + \frac{2}{\ell} \sum_{r=1}^\ell \hat u(r)\ .
\]
\end{lem}
\begin{proof}
The proof is the same as that of \cite[Lemma~2.1]{CG}, using $c_1=1$.
\end{proof}

Now we estimate characteristic functions. Our main goal is to show that
\begin{equation}\label{eq: main_goal}
\limsup_n \left| \phi_n\left( \frac{t}{\sigma(n)} \right) - e^{-t^2/2} \right| = 0\ .
\end{equation}
It will follow from three statements:
\begin{equation}\label{eq: equation_1}
\limsup_n \left| \phi_n \left( \frac{t}{\sigma(n)} \right) - \phi_{n,\ell} \left( \frac{t}{s(n,\ell)} \right) \right| \leq  |t| \frac{2}{\ell} \sum_{r=1}^\ell \hat u(r)\ ,
\end{equation}
\begin{equation}\label{eq: equation_2}
\limsup_n \left| \phi_{n,\ell}\left( \frac{t}{s(n,\ell)} \right) - \prod_j \phi_{n,j} \left( \frac{t}{s(n,\ell)} \right) \right| \leq \frac{t^2}{\ell} \sum_{r=1}^\ell \hat u(r)\ ,
\end{equation}
\begin{equation}\label{eq: equation_3}
\limsup_n  \left| \prod_{1 \leq j \leq m} \phi_{n,j}\left( \frac{t}{s(n,\ell)} \right) - e^{-\frac{t^2}{2}} \right| = 0\ .
\end{equation}
Indeed, combining these with the triangle inequality and taking $\ell \to \infty$, the assumption $\hat u(r) \to 0$ implies \eqref{eq: main_goal}. The proofs of \eqref{eq: equation_1} and \eqref{eq: equation_2} are exactly the same as those of (2.1) and (2.2) of \cite{CG}. The main tool is Lemma~\ref{lem: 2.1} along with the estimate
\[
\left| \mathbb{E}e^{itX} - \mathbb{E}e^{itY} \right| \leq |t| \|X-Y\|_2\ ,
\]
which holds for any variables $X,Y$ with two moments. These ideas suffice for \eqref{eq: equation_1}; for \eqref{eq: equation_2} one also needs the following lemma from \cite{NW}:
\begin{lem}[Newman-Wright]
Suppose $X_1, \ldots, X_m$ are associated finite variance random variables with joint and marginal characteristic functions $\psi(r_1, \ldots, r_m)$ and $\psi_j(r)$; then
\[
\left| \psi(r_1, \ldots, r_m) - \prod_{j=1}^m \psi_j(r_j) \right| \leq \frac{1}{2} \sum \sum_{1 \leq j \neq k \leq m} |r_j| |r_k| \Cov(X_j,X_k)\ .
\]
\end{lem}

So we move to \eqref{eq: equation_3}, which is the main difference in our setting. As in \cite{CG}, we use Lyapunov's theorem: we must show that
\[
\frac{1}{s(n,\ell)^{2+\delta}} \sum_{1 \leq j \leq m} \mathbb{E}|Y_j^{(n)}|^{2+\delta} \to 0 \text{ as } n \to \infty\ .
\]
By positive association, $s(n,\ell)^2 \geq m\ell$. This and stationarity imply that it suffices to show
\[
\frac{m}{(m\ell)^{1+\delta/2}} \mathbb{E}\left| \sum_{i=1}^\ell \eta_i^{(n)}\right|^{2+\delta} \to 0\ .
\]
However, the triangle inequality gives an upper bound of
\[
\ell^{1+\delta/2} \cdot \frac{1}{m^{\delta/2}} \|\eta_1^{(n)}\|_{2+\delta}^{2+\delta} = \ell^{1+\delta/2} \left( \frac{1}{m^{\frac{1}{2}-\frac{1}{2+\delta}}} \|\eta_1^{(n)}\|_{2+\delta} \right)^{2+\delta}\ .
\]
As $\ell$ is fixed and $\|\eta_1^{(n)}\|_{2+\delta} = o\left(r(n)^{\frac{1}{2}-\frac{1}{2+\delta}}\right)$, we obtain \eqref{eq: equation_3} and complete the proof.

%Putting together \eqref{eq: new_1} with \eqref{eq: 2.3}, for fixed $\ell$,
%\[
%\limsup_N \left| \phi_N\left( \frac{t}{\sigma(N)} \right) - e^{-\frac{t^2}{2}} \right| \leq \left[ \frac{2|t|}{\ell} + \frac{t^2}{\ell} \right] \sum_{r=1}^\ell u(r)\ .
%\]
%Because $u(r) \to 0$ as $r \to \infty$, we can take $\ell \to \infty$ to obtain \eqref{eq: main_goal} and complete the proof.

\bigskip
\noindent
{\bf Acknowledgements.} The authors thank Phil Sosoe for comments on a previous version of this manuscript. J. H. thanks R. Bradley for discussions related to CLT's for mixing arrays. M. D. thanks L. Zhu for pointing out CLT's for associated variables. The authors thank the organizers of the 2013 Midwest Probability Colloquium, where part of this work was completed.

%\bibitem{Sutav}
%S. Chatterjee. (2013). The universal relation between scaling exponents in first-passage percolation. \emph{Ann. Math.} {\bf 177}, 663--697.

\end{document}